\documentclass{amsart}

\usepackage{amsmath,amssymb}
\usepackage{dsfont}

\newcommand{\DD}{\Delta}
\newcommand{\R}{\mathbb{R}}
\newcommand{\C}{\mathbb{C}}
\newcommand{\N}{\mathbb{N}}

\newcommand{\E}{\mathcal{E}}
\renewcommand{\H}{\mathcal{H}}
\newcommand{\cN}{\mathcal{N}}

\renewcommand{\epsilon}{\varepsilon}
\renewcommand{\leq}{\leqslant}
\renewcommand{\geq}{\geqslant}

\renewcommand{\to}{\rightarrow}
\newcommand{\weakto}{\rightharpoonup}

\newcommand{\Rplus}{\mathbb{R}_{+}}
\newcommand{\eps}{\varepsilon}
\newcommand\1{{\ensuremath {\mathds 1} }}

\newcommand{\amax}{\alpha_{\mathrm{max}}}

\newtheorem{thm}{Theorem}
\newtheorem{lemma}{Lemma}
\newtheorem{corollary}{Corollary}
\newtheorem{prop}{Proposition}
\newtheorem{definition}{Definition}
\newtheorem{assumption}{Assumption}
\newtheorem{remark}{Remark}
\newtheorem*{convention}{Convention}
\newtheorem*{remarks}{Remarks}
\newtheorem*{remcom}{Remark}
\numberwithin{equation}{section}
\numberwithin{lemma}{section}
\numberwithin{prop}{section}
\numberwithin{corollary}{section}
\numberwithin{thm}{section}
\numberwithin{definition}{section}
\numberwithin{assumption}{section}
\numberwithin{remark}{section}

\begin{document}

\title[Uniqueness and Nondegeneracy]{Uniqueness and Nondegeneracy of \\ Ground States for $(-\DD)^s Q + Q - Q^{\alpha+1} =0$ in $\R$}

\author{Rupert L. Frank and Enno Lenzmann}

\address{Rupert L. Frank, Department of Mathematics, Princeton University, Washington Road, Princeton, NJ 08544, USA.}
\email{rlfrank@math.princeton.edu}
\address{Enno Lenzmann, Institute for Mathematical Sciences, University of Copenhagen, Universitetsparken 5, 2100 Copenhagen \O, Denmark.}
\email{lenzmann@math.ku.dk}

\maketitle

\begin{abstract}
We prove uniqueness of ground state solutions $Q = Q(|x|) \geq 0$ for the nonlinear equation
$$
(-\DD)^s Q + Q - Q^{\alpha+1}= 0 \quad \mbox{in $\R$}, 
$$
where $0 < s < 1$ and $0 < \alpha < \frac{4s}{1-2s}$ for $s < 1/2$ and $0 < \alpha < \infty$ for $s \geq 1/2$. Here $(-\DD)^s$ denotes the fractional Laplacian in one dimension. In particular, we generalize (by completely different techniques) the specific uniqueness result obtained by Amick and Toland for $s=1/2$ and $\alpha=1$ in [Acta Math., \textbf{167} (1991), 107--126].

As a technical key result in this paper, we show that the associated linearized operator  $L_+ = (-\DD)^s + 1 - (\alpha+1) Q^\alpha$ is nondegenerate; i.\,e., its kernel satisfies  $\mathrm{ker}\, L_+ = \mathrm{span} \, \{ Q' \}$. This result about $L_+$ proves a spectral assumption, which plays a central role for the stability of solitary waves and blowup analysis for nonlinear dispersive PDEs with fractional Laplacians, such as the generalized Benjamin-Ono (BO) and Benjamin-Bona-Mahony (BBM) water wave equations.
\end{abstract}

\section{Introduction}


\medskip

Fractional powers of the Laplacian arise in a numerous variety of equations in mathematical physics and related fields; see, e.\,g., \cite{CaSi07,AbBoFeSa89,We87,MaMcTa97,LiYa87,ElSc07,FrLe07,KeMaRo10} and references therein. Here, a central role within these models is often played by so-called {\em ground state solutions}, or simply {\em ground states}. By this, we mean nontrivial, nonnegative and radial functions $Q = Q(|x|) \geq 0$ that vanish at infinity and satisfy (in the distributional sense) an equation of the form
\begin{equation} \label{eq:GS}
(-\DD)^s Q + F(Q)  = 0 \quad \mbox{in $\R^d$}.
\end{equation} 
As usual, the fractional Laplacian $(-\DD)^s$ with $0 < s< 1$ is defined via its multiplier $|\xi|^{2s}$ in Fourier space, whereas $F(Q)$ denotes some given nonlinearity. In most examples of interest, the existence of ground states $Q= Q(|x|) \geq 0$ follows from variational arguments, applied to a suitable minimization problem whose Euler-Lagrange equation is given by \eqref{eq:GS}. Moreover, based on this variational approach, it is natural in these cases to require that a ground state is also a minimizer for some related variational problem in addition to just being a nonnegative and radial solution of \eqref{eq:GS}. Indeed, we will make use of this (strengthened) notion of a ground state in this paper further below.

\medskip
In striking contrast to the question of existence, it seems fair to say that extremely little is known about uniqueness of ground states $Q=Q(|x|) \geq 0$ for problems like \eqref{eq:GS}, except for the ``classical'' limiting case with $s=1$, where standard ODE methods are applicable.  Indeed, to the best of the authors' knowledge, the only examples for which uniqueness of ground states for \eqref{eq:GS} has been proven are: 
\begin{itemize}
\item Ground state solitary waves for the Benjamin-Ono equation in $d=1$ dimension; see \cite{AmTo91}. 
\item Optimizers for  fractional Sobolev inequalities in $d \geq 1$ dimensions; see \cite{ChLiOu06, Li04}.
\end{itemize}
In fact, in both cases the unique ground states are known in closed form. However, the uniqueness proof of both results hinges on a very specific feature of each problem: In the first case, the proof is intimately linked to complex analysis and special identities exhibited by the (completely integrable) Benjamin-Ono equation; whereas in the second case, the conformal symmetry of Sobolev inequalities plays a key role in the uniqueness proof. In particular, the specific arguments developed in \cite{AbBoFeSa89,AmTo91, ChLiOu06,Li04} are apparently of no use in a more general setting. Hence, we see that a satisfactory understanding of uniqueness for ground states of problems like \eqref{eq:GS} is largely missing. Clearly, the main analytical obstruction is that shooting arguments and other ODE techniques (which are essential in the classical case $s=1$; see, e.\,g., \cite{Kw89,McSe87,Mc93}) are not applicable to the nonlocal operator $(-\DD)^s$ when $0 < s < 1$.

\medskip
In the present paper, we address the question of uniqueness for a general class of the form \eqref{eq:GS} in $d=1$ space dimension. More precisely, we prove uniqueness of ground states $Q \in H^s(\R)$ for the nonlinear model problem
\begin{equation} \label{eq:phi}
(-\DD)^s Q + Q - Q^{\alpha+1}   = 0 \quad \mbox{in $\R$}.
\end{equation}
Here we assume that $0 < s <1$ and $0 < \alpha < \amax(s)$ holds,  where the critical exponent $\amax(s)$ is defined as
\begin{equation}
\amax(s) := \left \{ \begin{array}{ll} \frac{4s}{1-2s} & \quad \mbox{for $0 < s < \frac{1}{2}$}, \\
+\infty & \quad \mbox{for $\frac{1}{2} \leq s < 1$.}  \end{array} \right . 
\end{equation} 
Technically speaking, the condition that $\alpha $ be strictly less than $\amax(s)$, which is vacuous if $s \geq \frac{1}{2}$, ensures that the nonlinearity in equation \eqref{eq:phi} is {\em $H^s$-subcritical}. In fact, it turns out that this condition on $\alpha$ is necessary to have existence of ground states for \eqref{eq:phi}, since (by so-called Pohozaev identities) it is easy to see that \eqref{eq:phi} does not admit any non-trivial solutions in $H^s(\R) \cap L^{\alpha+2}(\R)$ when $\alpha \geq \amax(s)$ holds.

\medskip
Apart from being a natural model case for equation \eqref{eq:GS} in one space dimension, we remark that equation \eqref{eq:phi} and its solutions provide {\em solitary wave solutions} for three fundamental nonlinear dispersive model equations in $d=1$ dimensions: The generalized Benjamin-Ono equation (gBO), the  Benjamin-Bona-Mahony equation (gBBM) and the fractional nonlinear Schr\"odinger equation (fNLS):
\begin{equation*} \tag{gBO} \label{gBO}
u_t + u _x  - ((-\DD)^s u)_x +  u^{\alpha} u_x= 0, 
\end{equation*}
\begin{equation*} \tag{gBBM} \label{gBBM}
u_t + u _x + ((-\DD)^s u)_t +  u^{\alpha} u_x = 0,
\end{equation*}
\begin{equation*} \tag{fNLS}
i u_t - (-\DD)^s u + |u|^{\alpha} u = 0 , 
\end{equation*}
Note that in (gBO) and (gBBM) we assume that $\alpha \in \mathbb{N}$ is an integer and that $u = u(t,x)$ is real-valued.\footnote{We could extend to complex-valued $u$ and non-integer $\alpha$, by replacing $u^\alpha u_x$ with $|u|^\alpha u_x$. Indeed, such models are also of interest in the PDE literature; see, e.\,g., \cite{KeMaRo10}.} Suppose now that $Q=Q(|x|) \geq 0$ solves \eqref{eq:phi}. Then it is elementary to see that the following functions 
$$
u_c(t,x) = (c(\alpha+1))^{\frac{1}{\alpha}} Q\big (c^{\frac{1}{2s}}( x -(1+c)t) \big ),
$$
$$
u_c(t,x) = (c (\alpha+1))^{\frac{1}{\alpha}} Q \big ( (\frac{c}{1+c})^{\frac{1}{2s}} (x-(1+c)t ) \big ) ,
$$
$$
u_\omega(t,x) = e^{i  \omega t} \omega^{\frac{1}{\alpha}} Q(\omega^{\frac{1}{2s}} x),
$$
provide solitary wave solutions for (gBO), (gBBM) and (fNLS), respectively. In the first two (water wave) examples, the parameter $c > 0$ corresponds to the traveling speed of the wave to the right; whereas the parameter $\omega > 0$ plays the role of a oscillation frequency of the solitary wave for (fNLS). We refer to, e.\,g., \cite{We87,KeMaRo10,AlBo91, Lin08, BeBr83} for results on solitary waves for (gBO), (gBBM) and (fNLS). 

In all these cases, the uniqueness and the so-called nondegeneracy (see below) of the ground states $Q=Q(|x|) \geq 0$ are of fundamental importance in the stability and blowup analysis for the corresponding solitary waves $u_c(t,x)$ and $u_\omega(t,x)$ above. So far, except for the special case $s=1/2$ and $\alpha=1$ in \cite{AmTo91} and a perturbative result for $s \approx 1$ in \cite{KeMaRo10}, no rigorous results have been derived in this direction, and hence these properties of $Q=Q(|x|)$ have been imposed in terms of assumptions, partly supported by numerical evidence. In Theorems \ref{thm:nondeg} and \ref{thm:unique} below, we will in fact resolve uniqueness and so-called nondegeneracy of ground states for equation \eqref{eq:phi} in the full range $0 < s <1$ and $0 < \alpha < \amax(s)$.

\medskip
Before we formulate the main results of this paper, let us first recall some facts  about existence, regularity and spatial decay of ground state solutions for equation \eqref{eq:phi}. Indeed, by following the seminal approach of M.~Weinstein in \cite{We85,We87}, we notice that problem \eqref{eq:phi} has indeed non-trivial solutions $Q \in H^s(\R)$, which are optimizers of the Gagliardo-Nirenberg type inequality
\begin{equation}
\label{ineq:GN}
\int_{\R} | u |^{\alpha+2} \leq C_{\alpha,s} \left ( \int_{\R} | (-\Delta)^{\frac s 2} u |^{2} \right )^\frac{\alpha}{4s} \left ( \int_{\R} |u |^2 \right )^{\frac{\alpha}{4s} (2s - 1) + 1} ,
\end{equation}
in one space dimension. Here $C_{\alpha,s} > 0$ denotes the optimal constant depending on $\alpha$ and $s$. Equivalently, this claim follows from considering the `Weinstein' functional
\begin{equation} \label{eq:Jdef}
J^{s,\alpha}(u) := \frac{ \left ( \int | (-\Delta)^{\frac s 2} u |^{2} \right )^\frac{\alpha}{4s} \left ( \int |u |^2 \right )^{\frac{\alpha}{4s} (2s - 1) + 1} }{\int |u|^{\alpha+2} },
\end{equation}
defined for $u \in H^{s}(\R)$ with $u \not \equiv 0$. Clearly, every minimizer $Q\in H^s(\R)$ for $J^{s,\alpha}(u)$ optimizes the interpolation estimate \eqref{ineq:GN} and vice versa. In addition, any such nonnegative $Q \in H^s(\R)$ is found to satisfy equation \eqref{eq:phi} after some suitable rescaling $Q(x) \mapsto a Q(bx)$ with some positive constants $a > 0$ and $b > 0$.  

In summary, we have the following existence result and fundamental properties of solutions to \eqref{eq:phi}, which we can  infer from the literature.

\begin{prop} \label{prop:Q}
Let $0 < s < 1$ and $0 < \alpha < \amax(s)$. Then the following holds.

\begin{itemize}
\item[(i)] {\bf Existence:} There exists a solution $Q \in H^s(\R)$ of equation \eqref{eq:phi} such that $Q = Q(|x|) > 0$ is even, positive and strictly decreasing in $|x|$. Moreover, the function $Q \in H^s(\R)$ is a minimizer for  $J^{s,\alpha}(u)$.

\item[(ii)] {\bf Symmetry and Monotonicity:} If $Q \in H^s(\R)$ with $Q \geq 0$ and $Q \not \equiv 0$ solves \eqref{eq:phi}, then there exists $x_0 \in \R$ such that $Q(\cdot - x_0)$ is an even, positive and strictly decreasing in $|x-x_0|$.  

\item[(iii)] {\bf Regularity and Decay:} If $Q \in H^s(\R)$ solves \eqref{eq:phi}, then $Q \in H^{2s+1}(\R)$. Moreover, we have the decay estimate
$$
|Q(x)| + |xQ'(x)| \leq \frac{C}{1+|x|^{1+2s}} 
$$
 for all $x \in \R$ and some constant $C > 0$.
\end{itemize}
\end{prop}  

\begin{remarks}
{\em 
1.) As for the proof of Part (i), we can refer to M.~Weinstein's paper \cite{We87} where concentration-compactness type arguments are used to show existence of minimizers for $1/2 \leq s <1$. But the method can be applied to the range $0 < s < 1/2$ as well. Moreover, by {\em strict rearrangement inequalities} for $\int | (-\DD)^\frac{s}{2} u |^2$ when  $0 < s <1$ (see, e.\,g., \cite{FrSe08}), we can deduce that any minimizer $Q \in H^s(\R)$ for $J^{s,\alpha}(u)$ must be equal (apart from translation and phase) to its symmetric-decreasing rearrangement $Q^* = Q^*(|x|)$. See also Sections \ref{sec:main} and \ref{sec:unique} below.

2.) To derive Part (ii), we can directly adapt the moving plane method  recently developed in \cite{MaZh10}, combined with some properties of the integral kernel for $((-\DD)^s+1)^{-1}$ on $\R$. For more details, we refer to Appendix \ref{app:regularity}.

3.) The regularity proof of Part (iii) is worked out in Appendix \ref{app:regularity}. Moreover, it easy to see that $Q \in H^k(\R)$ for all $k \geq 1$, if the exponent $\alpha$ is a positive integer; see also \cite{LiBo96} for an analyticity result in this case. See \cite{KeMaRo10} and references given there for the spatial decay estimate stated above. 
}

\end{remarks}

\section{Main Results}

\label{sec:main} 

We now formulate the main results of this paper about ground state solutions to \eqref{eq:phi} that we define as follows.

\begin{definition} \label{def:gs}
Let $Q \in H^s(\R)$ be an even and positive solution of \eqref{eq:phi}. If $$J^{s,\alpha}(Q) = \inf_{u \in H^s(\R) \setminus \{0\}} J^{s,\alpha}(u),$$ then we say that $Q \in H^s(\R)$ is a {\bf ground state solution} of equation \eqref{eq:phi}.
\end{definition}

\begin{remark}
{\em In fact, there are several constrained variational problems that are equivalent to the unconstrained problem of minimizing $J^{s,\alpha}(u)$ on $H^s(\R) \setminus \{0\}$.

For example, in the so-called $L^2$-subcritical case when $0 < \alpha < 4s$, the constrained minimization problem, with parameter $N >0$,
$$
E(N) = \inf \left  \{ \frac{1}{2} \int |(-\DD)^{\frac{s}{2}} u |^2 + \frac{1}{\alpha+2} \int  |u|^{\alpha+2} : u \in H^s(\R), \int |u|^2 = N \right  \}
$$   
is attained at $u \in H^s(\R)$ if and only if $u=e^{i \vartheta} \lambda^{\frac{1}{\alpha}} Q( \lambda^{\frac{1}{2s}} ( \cdot + y ))$ with some $\vartheta \in \mathbb{R}$, $y \in \R$ and $\lambda >0$ chosen to ensure that $\int |u|^2 = N$ holds. Here $Q \in H^s(\R)$ is a ground state solution of \eqref{eq:phi} in the sense of Definition \ref{def:gs} above.
}
\end{remark}

Our first main result establishes the so-called {\em nondegeneracy} of the linearization associated with positive solutions $Q$ of \eqref{eq:phi} that are {\em local} minimizers for $J^{s, \alpha}(u)$; thus our result holds in particular when $Q$ is a ground state solution. As already mentioned, the nondegeneracy of the linearization around ground states plays a fundamental role  in the stability and blowup analysis for solitary waves for related time-dependent equations such as the generalized (gBO) and (gBBM) equations; see, e.\,g., \cite{We87, KeMaRo10, AlBo91, Lin08, BeBr83}, where the nondegeneracy of $L_+$ is imposed in terms of a spectral assumption, or proven for $s$ close to 1 by perturbation arguments. 

We have the following general nondegeneracy result.


\begin{thm}{\bf (Nondegeneracy).} \label{thm:nondeg}
Let $0 < s < 1$ and $0 < \alpha < \amax(s)$. Suppose that $Q \in H^s(\R)$ is a positive solution of \eqref{eq:phi} and consider the linearized operator
$$
L_+ = (-\DD)^s + 1 - (\alpha+1) Q^{\alpha}
$$
acting on $L^2(\R)$. Then the following conclusion holds: If $Q$ is local minimizer for $J^{s,\alpha}(u)$, then $L_+$ is nondegenerate, i.\,e., its kernel satisfies
$$
\mathrm{ker} \, L_+ = \mathrm{span} \, \{ Q' \}. 
$$
In particular, any ground state solution $Q= Q(|x|) > 0$ of equation \eqref{eq:phi} has a nondegenerate linearized operator $L_+$.
\end{thm}

\begin{remarks} {\em 1.) In fact, we will prove the nondegeneracy of $L_+$ under the (weaker) second-order condition 
$$
\frac{d^2}{d \eps^2} \Big |_{\eps = 0} J^{s,\alpha}(Q + \eps \eta) \geq 0  \quad \mbox{for all $\eta \in C^\infty_0(\R)$},
$$
which clearly holds true when $Q \in H^s(\R)$ is a local minimizer for $J^{s,\alpha}(u)$.

2.) An important application of Theorem \ref{thm:nondeg} arises in the stability and blowup analysis of solitary waves for related time-dependent equations; notably in terms of a  {\em coercivity estimate} for $L_+$, which readily follows from the nondegeneracy of $L_+$.  More precisely, for suitable two-dimensional subspaces $M \subset L^2(\R)$, we can derive the lower bound
$$
\langle \eta, L_+ \eta \rangle \geq \delta \| \eta \|_{H^s}^2 \quad \mbox{for $\eta \perp M$,}
$$
where $\delta > 0$ is some positive constant independent of $\eta$. For example, by using the result of Theorem \ref{thm:nondeg}, it is to easy see that $M = \mathrm{span} \, \{ \phi, Q' \}$ is a suitable choice, where $\phi = \phi(x)$ denotes the first eigenfunction of $L_+$ acting on $L^2(\R)$. 
}
\end{remarks}

Let us briefly comment on the proof of Theorem \ref{thm:nondeg}. The essential idea of the proof is to find to a suitable substitute for Sturm-Liouville theory in order to estimate the number of sign changes for the {\em second eigenfunction(s)} for ``fractional'' Schr\"odinger operators of the form 
$$
H=(-\DD)^s + V
$$  
in $d=1$ space dimension. In fact, it turns out that a key step in the proof of Theorem \ref{thm:nondeg} follows from an argument in \cite{ChGuNaTs07} developed for the classical ODE case when $s=1$ holds, {\em provided} we know that any (even) second eigenfunction of $L_+$ can change its sign only once on the positive real line $\{x > 0\}$. Obviously, the crux of that matter is that $(-\DD)^s$ is a nonlocal operator when $0 < s <1$; and hence estimating the number of zeros for eigenfunctions of $H=(-\DD)^s + V$ requires new arguments and insights, which substitute classical ODE techniques. 

Let us briefly explain in general terms how we tackle this difficulty. First, we recall the known fact that $(-\DD)^s$ can be regarded as a Dirichlet-Neumann operator for a suitable elliptic problem on the upper halfplane $\R^2_+  = \{ (x,y) \in \R^2 : y > 0 \}$; see, e.\,g., the recent work by Caffarelli-Silvestre in \cite{CaSi07} and also Graham-Zworski in \cite{GrZw03} for this observation in a geometric context. Using now this extension to the upper halfplane $\R^2_+$,  we derive --- as a technical key result --- a variational characterization of the eigenfunctions (and eigenvalues) for fractional Schr\"odinger operators $H=(-\DD)^s+V$ in terms of the Dirichlet type functional
$$
 A(u) = \iint_{\R_+^2} |\nabla u(x,y)|^2 y^{1-2s} \, dx \,dy + \int_{\R} V(x) |u(x,0)|^2 \, dx ,
$$
which is defined for suitable class of functions $u=u(x,y)$ on the upper halfplane $\R^2_+$, where $u(x,0)$ denotes the trace of $u(x,y)$ on the boundary $\partial \R_+^2 = \R \times \{ 0 \}$. Moreover, for the variational problem based on the functional $A(u)$, we establish a nodal domain bound \`a la Courant. From such estimates we can finally deduce a sharp upper bound on the number of sign changes for any second eigenfunction of the nonlocal operator $H=(-\DD)^s +V$ acting on $L^2(\R)$, as needed in the proof of Theorem \ref{thm:nondeg}. Furthermore, this estimate for eigenfunctions for $H$ involving $(-\DD)^s$ can be viewed as a generalization of the inspiring work by Ba\~nuelos-Kulczycki in \cite{BaKu04}, which studies eigenfunctions for $\sqrt{-\DD}$ on bounded intervals in $\R$.

\medskip
We now turn to the second main result of this paper, which proves global uniqueness of ground state solutions. As a consequence, we also obtain uniqueness of optimizers  for the interpolation inequality \eqref{ineq:GN}  up to scaling and translations.

\begin{thm} {\bf (Uniqueness).} \label{thm:unique}
Let $0 < s < 1$ and $0 < \alpha < \amax(s)$. Then the ground state solution $Q=Q(|x|) > 0$ for equation \eqref{eq:phi} is unique.

Furthermore, every optimizer $v \in H^s(\R)$ for the Gagliardo-Nirenberg inequality \eqref{ineq:GN} is of the form $v = \beta Q( \gamma( \cdot + y))$ with some $\beta \in \C$, $\beta \neq 0$, $\gamma > 0$ and $y \in \R$.
\end{thm}

\begin{remarks} {\em 1.) Under the assumption that $Q = Q(|x|) > 0$ minimizes $J^{s,\alpha}(u)$, we remark that Theorem \ref{thm:unique} generalizes the striking result  by Amick and Toland in \cite{AmTo91} about uniqueness of positive solutions $Q=Q(|x|) > 0$ that satisfy
\begin{equation} \label{eq:BO}
(-\DD)^{\frac{1}{2}} Q + Q - Q^2 = 0 \quad \mbox{in $\R$}.
\end{equation} 
In fact, it was proven in \cite{AmTo91} that (apart from translations) the function $$Q(x) = \frac{2}{1+x^2}$$ is the only positive of $\eqref{eq:BO}$ in $H^{\frac{1}{2}}(\R)$. However, the remarkably elegant approach taken in \cite{AmTo91} makes essential use of  complex analysis (e.\,g., harmonic conjugates and Cauchy-Riemann equations) in combination with very specific identities derived from \eqref{eq:BO}. In particular, it appears to be a hopeless enterprise to try to generalize the arguments in \cite{AmTo91} to different powers of the  fractional Laplacians $(-\DD)^s$ with $s \neq 1/2$ or non-quadratic nonlinearities $f(Q) = Q^{\alpha+1}$ with $\alpha \neq 1$.

2.)  The uniqueness of optimizers for interpolation inequality \eqref{ineq:GN} follows directly from the ground state uniqueness and the strict rearrangement inequalities in \cite{FrSe08}; see also Section \ref{sec:unique} below.   
}
\end{remarks}

Let us briefly explain the strategy behind the proof of the ground state uniqueness result of Theorem \ref{thm:unique}. First, we fix $0 < s_0 < 1$ and $0 < \alpha < \amax(s_0)$ and suppose that $Q_0 = Q_0(|x|) > 0$ is a ground state solution to \eqref{eq:phi} with $s=s_0$. By the nondegeneracy result of Theorem \ref{thm:nondeg}, the associated linearized operator $L_+$ is invertible on $L^2_{\mathrm{even}}(\R) \perp \mathrm{ker} \, L_+$. Hence, by using an implicit function argument, we can construct around $(Q_0,1)$ a locally unique branch of solutions $(Q_s, \lambda_s)$ (in some suitable Banach space) which satisfy
\begin{equation}
 (-\DD)^s Q_s + \lambda_s Q_s - |Q_s|^\alpha Q_s = 0
 \end{equation}
where $s \in [s_0, s_0+\delta)$ and $\delta > 0$ being sufficiently small. Here the function $\lambda_s$ is introduced to ensure that the conservation law\footnote{Equivalently, we could keep $\lambda_s \equiv 1$ at the expense of varying $\int |Q_s|^{\alpha+2}$. However, it turns out that using $\lambda_s$ is convenient when we derive a-priori bounds for $Q_s$.}
$$
\int |Q_{s}|^{\alpha+2} = \int |Q_0|^{\alpha+2}
$$
holds along the branch $(Q_s, \lambda_s)$. Furthermore, we show that positivity is preserved along the branch, i.\,e., we have $Q_s = Q_s(|x|) >0$ for all $s \in [s_0, s_0 +\delta)$ thanks to $Q_0 = Q_0(|x|) > 0$ initially.  Note that, although we start from a ground state solution for $s=s_0$, it {\em cannot} be inferred that $Q_s$ (up to a rescaling) is also a ground state solution; i.\,e., global minimizers for $J^{s,\alpha}(u)$. Therefore, the global continuation of the branch $(Q_s,\lambda_s)$ to $s=1$, say, is far from obvious. 

However, as an essential step in the uniqueness proof, we show that the branch $(Q_s, \lambda_s)$ can be {\em indeed} continued for all $s \in [s_0,1)$. This global continuation will be based on the nondegeneracy result of Theorem \ref{thm:nondeg} in combination with the a-priori bounds 
$$
1 \lesssim \int | (-\DD)^{\frac{s}{2}} Q_s |^2 \lesssim 1 ,  \quad 1 \lesssim \int | Q_s |^2 \lesssim 1 , \quad 1 \lesssim \lambda_s \lesssim 1. 
$$
Here it turns out that establishing the upper bound $\int  | Q_s |^2 \lesssim 1$ is the most delicate step and thus it requires a careful analysis of the problem. In addition to a-priori regularity bounds, the strict positivity and monotonicity of $Q_s = Q_s(|x|) >0$ also enters in a significant way, since it allows us to derive the uniform decay estimate $Q_s(|x|) \lesssim |x|^{-1}$ for $|x| \gtrsim 1$. The latter fact then guarantees relative compactness of $\{ Q_s \}$ in certain $L^p$-norms. 

Once we have established that $(Q_s, \lambda_s)$ can be extended to $s=1$, we conclude that $Q_s \to Q_*$ (in some suitable sense) and $\lambda_s \to \lambda_*$ as $s \to 1$, where $Q_* = Q_*(|x|) >0$ and $\lambda_* > 0$ satisfy
$$
-\DD Q_* + \lambda_* Q_* - Q_*^{\alpha+1} = 0.
$$
For this limiting equation, it is well-known (by standard ODE techniques) that uniqueness of even and positive solutions $Q_*=Q_*(|x|) > 0$ holds true. Furthermore, by Pohozaev type identities and the conservation law for $\int |Q_s|^{\alpha+2}$ and the fact that $Q_0$ is a ground state, we deduce that the limit $\lambda_* = \lambda_*(s_0, \alpha)$  only depends on $s_0$ and $\alpha$. Hence, we can conclude that two different branches $(Q_s, \lambda_s)$ and $(\widetilde{Q}_s, \widetilde{\lambda}_s)$ (both starting from a ground state with $s=s_0$) must converge to the same limit $(Q_*, \lambda_*)$. By using the known nondegeneracy for the linearization around $(Q_*,\lambda_*)$, we can infer that the branches $(Q_s, \lambda_s)$ and $(\widetilde{Q}_s, \widetilde{\lambda}_s)$ must intersect for some $s \in [s_0,1)$ in contradiction to the local uniqueness of branches. This fact establishes  uniqueness of ground states for all $0 < s_0 < 1$ and $0 < \alpha < \amax(s_0)$, as stated in Theorem \ref{thm:unique}.

Finally, we mention that the second part of Theorem \ref{thm:unique} follows from the fact that every optimizer for \eqref{ineq:GN} must be equal to its symmetric-decreasing rearrangement modulo scaling and translation. The proof of this will be mainly based on strict rearrangement inequalities for $(-\DD)^s$.  

\subsection*{Extension of Main Results to Higher Dimensions} 
With regard to possible extensions to higher dimensions, we remark that most of the arguments presented here can be easily generalized to $d \geq 2$ dimensions. However, as the only notable exception, the proof of the oscillation estimate for the second eigenfunction $L_+$ (see Theorem \ref{thm:courant} below) hinges on the fact that $L_+$ acts on functions in $d=1$ dimension. How to obtain a similar oscillation estimate for radial eigenfunctions of $L_+$ in $d \geq 2$ dimensions remains the chief open problem. If this could be solved, the analogous nondegeneracy result of Theorem \ref{thm:nondeg} would readily follow for $d \geq 2$ dimensions. Moreover, the uniqueness proof of Theorem \ref{thm:unique} would allow for an straightforward adaption to $d \geq 2$ dimensions, since we have uniqueness and nondegeneracy of positive radial solutions $Q= Q(|x|) > 0$ in $H^1(\R^d)$ for the limiting equation
$$
-\DD Q + Q - Q^{\alpha+1} = 0  \quad \mbox{in $\R^d$}, 
$$ 
where $0 < \alpha < \infty$ for $d=2$ and $0 < \alpha < \frac{4}{d-2}$ for $d \geq 3$; see, e.\,g., \cite{Kw89}.

\subsection*{Plan of the Paper}
We organize this paper as follows. In Section \ref{sec:nodal}, we establish (as a technical key fact) a variational principle for fractional Schr\"odinger operators $H = (-\DD)^s + V$ in terms of a {\em local energy functional}. As a main consequence, we obtain a sharp bound on the number of sign changes for any second eigenfunction of $H$. Then in Section \ref{sec:nondeg}, we prove Theorem \ref{thm:nondeg}. Here we will make essential use of the main result from Section \ref{sec:nodal}. Finally, in Section \ref{sec:unique}, we establish the uniqueness of ground states as stated in Theorem \ref{thm:unique}. The proof will be based on the nondegeneracy result of Theorem \ref{thm:nondeg}, combined with an elaborate global continuation argument. The Appendix contains several technical results and proofs needed in this paper. 

\subsection*{Notation}
Throughout this paper, we employ standard notation for $L^p$-spaces and Sobolev spaces $H^s(\R)$ of order $s \in \R$. We use $\langle f, g \rangle = \int \overline{f} g $ to denote the inner product on $L^2(\R)$. (In fact, we will mostly deal with real-valued functions and hence complex conjugation is redundant.) Furthermore, we make the usual abuse of notation by writing both $f=f(x)$ and $f=f(|x|)$ whenever $f : \R \to \R$ is an even function. The (open) positive real axis will be denoted by $\R_+ = (0,\infty)$. Also, we use the standard notation $$X \lesssim  Y$$ to denote $X \leq C Y$ for some constant $C >0$ that only depends on some fixed quantities. Sometimes we write $X \lesssim_{a,b,\ldots} Y$ to underline that $C$ depends on the fixed quantities $a, b, \ldots$ etc.

\subsection*{Acknowledgments} R.\,F.~acknowledges support from NSF grant PHY-0652854. E.\,L.~was supported by a Steno fellowship from the Danish science research council, and he also gratefully acknowledges partial support from NSF grant DMS-0702492.


\section{An oscillation estimate for $H = (-\DD)^s + V$}
\label{sec:nodal}

This section serves as a preliminary discussion for Section \ref{sec:nondeg}, where we prove the nondegeneracy result of Theorem \ref{thm:nondeg}. More precisely, the present section deals with ``fractional'' Schr\"odinger operators
\begin{equation*}
H = (-\DD)^s + V 
\end{equation*}
acting on $L^2(\R)$. As our key technical result in this section, we prove a sharp bound on the number of sign changes for the {\em second eigenfunction(s)} of the nonlocal operator $H$, which will be formulated in Theorem \ref{thm:courant} below. The proof will be based on a variational characterization of the eigenvalues for $H$ in terms of a {\em local energy functional} and associated nodal domain bound \`a la Courant; see Corollary \ref{varprinc} and Theorem \ref{nodal} below.

Let us first introduce a suitable class of potentials $V$ for the fractional Schr\"odinger operators discussed here. In many respects (e.\,g., perturbation theory and properties of eigenfunctions), the following {\em ``Kato class''} (denoted by $K_s$) is a natural choice.
 
\begin{definition} \label{def:kato}
Let $0 < s  < 1$. We say that the potential $V \in K_s$ if and only if $V : \R \to \R$ is  measurable and satisfies
$$
\lim_{E \to \infty} \left \| ((-\DD)^s +E)^{-1} |V| \right \|_{L^\infty \to L^\infty} = 0. 
$$ 
\end{definition}

\begin{remarks}
{\em
1.) If $V \in K_s$, then $H=(-\Delta)^s +V$ defines a unique self-adjoint operator on $L^2(\R)$ with form domain $H^s(\R)$, and the corresponding heat  kernel $e^{-tH}$ maps $L^2(\R)$ into $L^\infty(\R) \cap C^0(\R)$ for any $t > 0$. In particular, any $L^2$-eigenfunction of $H$ is continuous and bounded. See also \cite{CaMaSi90} for equivalent definitions of $K_s$ and further background material.

2.) If $V \in K_s$, then $V$ is relatively form-bounded (with relative bound less than 1) with respect to $(-\DD)^s$. That is, for every $0 < \eps < 1$, there is a constant $C_\eps > 0$ such that
$$
\langle \psi, |V| \psi \rangle \leq \eps \langle \psi, (-\DD)^s \psi \rangle + C_\eps \langle \psi, \psi \rangle ,
$$
for all $\psi \in H^s(\R)$.  In fact, the latter condition is also sufficient for $V$ to be in $K_s$, provided that $C_\eps$ depends on $\eps$ in some explicit way. 

3.) In terms of $L^p$-spaces, we can derive the following useful criterion for real-valued $V$ to be in $K_s$. In fact, we have the following.
\begin{itemize}
\item If $0 < s \leq 1/2$ and $V \in L^p(\R)$ for some $p > 1/2s$, then $V \in K_s$.
\item If $1/2 < s < 1$ and $V \in L^p(\R)$ for some $p \geq 1$, then $V \in K_s$.
\end{itemize}
See Lemma \ref{lem:katoclass} for further details on this sufficient condition.
}
\end{remarks}

Let us now assume that $V \in K_s$ holds. Suppose that $\psi$ is an $L^2$-eigenfunction of $H= (-\DD)^s +V$.  Then, by the previous remark, we have that $\psi$ is a continuous and bounded function on $\R$. Note also that we can always assume that $\psi$ is real-valued, since $H=(-\DD)^s + V$ is a real operator (i.\,e., it preserves real and imaginary parts). In particular, we can define what it means that $\psi(x)$ changes its sign $N$ times.

\begin{definition}
Let $\psi \in C^0(\R)$ be real-valued and let $N \geq 1$ be an integer. We say that $\psi(x)$ changes its sign $N$ times if there exist points $x_1 < \ldots < x_{N+1}$ such that $\psi(x_i) \neq 0$ for $i = 1, \ldots,N+1$ and $\mathrm{sign} (\psi(x_i)) = -\mathrm{sign} (\psi(x_{i+1}))$ for $i=1, \ldots, N$.
\end{definition} 

\begin{remark}{\em 
For $\psi \in C^0(\R)$, we can define the {\em nodal domains} of $\psi(x)$ as the connected components of the open set $\{ x \in \R : \psi(x) \neq 0 \}$. If $\psi(x)$ cannot vanish to second order, then clearly the maximal number of sign changes of $\psi(x)$ equals $K-1$, where $K$ is the number of nodal domains of $\psi$. But in what follows, we prefer to work with the weaker notion of sign changes of $\psi(x)$.}
\end{remark}

\noindent
We are now ready to state the following main result of this section.

\begin{thm}{\bf (An oscillation estimate for $H$).} \label{thm:courant}
Let $0 < s < 1$, $V \in K_s$, and consider $H=(-\DD)^s +V$ acting on $L^2(\R)$. Suppose that $\lambda_2 < \inf \sigma_{\mathrm{ess}}(H)$ is the second eigenvalue of $H$ and let $\psi_2 \in H^s(\R) \cap C^0(\R)$ be a corresponding real-valued eigenfunction. Then $\psi_2 = \psi_2(x)$ changes its at most twice on $\R$.

In particular, if $\psi_2 = \psi_2(|x|)$ is an even eigenfunction, then $\psi_2(|x|)$ changes its sign at exactly once on the positive axis $\{ x > 0 \}$.
\end{thm}

\begin{remarks}{\em 
1.) The reader who is mainly interested in applying this technical result may fast forward to Section \ref{sec:nondeg} at first reading.

2.) By Perron-Frobenius arguments (see Appendix \ref{app:kato}) the first eigenfunction $\psi_1=\psi(x) > 0$ of $H$ is always strictly positive. Hence, by the self-adjointness of $H$, we easily see that $\psi_2$ changes its sign at least once in order to satisfy the orthogonality condition $\langle \psi_1, \psi_2 \rangle = 0$. }
\end{remarks}

The proof of Theorem \ref{thm:courant} will be given at the end of this section. But first we have to establish some auxiliary facts in the following subsections. In particular, we derive the key variational principle of eigenvalues of $H$ in terms of a {\em local energy functional}, which we formulate in Corollary \ref{varprinc} below.


\subsection{Extension to $\R^2_+$ and a Sharp Trace Inequality}

We recall the known fact that the fractional Laplacian $(-\DD)^s$ on $\R^d$ can be expressed as the Dirichlet-to-Neumann operator for a suitable \emph{local} problem on the upper halfspace $\R^{d+1}_+ = \{ (x,y) : x \in \R^d, \; y  > 0\}$. See the recent work by Caffarelli-Silvestre \cite{CaSi07} for this fact. We also refer to the work of Graham-Zworski \cite{GrZw03}, where this observation occurred in a geometric context; see \cite{ChGo10} for a comparison and extension of  \cite{CaSi07} and \cite{GrZw03}.

We consider $d=1$ space dimension in the sequel. Let $0 < s < 1$ be given and set $a = 1-2s$. For a measurable function $f : \R \to \R$, we (first formally) define its extension $\E_a f : \R^2_+ \to \R$ as
\begin{equation}
\label{eq:econv}
(\E_{a} f)(x,y) := \int_{\R} y^{-1} P_{a}((x-x')/y) f(x') \, dx' ,
\end{equation}
where the convolution kernel $P_a : \R \to \R$ is given by
\begin{equation}
P_a(x) := \frac{\Gamma(\frac{2-a}{2})}{\pi^{1/2} \Gamma(\frac{1-a}{2})}  \frac{1}{(1+ x^2)^{\frac{2-a}{2}}}\,.
\end{equation}
Under suitable assumptions on $f$ it is known (see, e.\,g., \cite{CaSi07}) that $w = \E_a f$ solves the boundary value problem
\begin{equation} \label{eq:divw}
\left \{ \begin{array}{ll} \mathrm{div}(y^a \nabla w) = 0 &  \quad \mbox{in $\R^2_+$}, \\
w = f& \quad \mbox{on $\partial \R^2_+ = \R \times \{0\}$} .
\end{array} \right .
\end{equation}
Here the boundary condition $w=f$ is understood in some suitable sense, which will be formulated below. If $f$ is sufficiently regular, then we also have that
\begin{equation} \label{eq:boundary}
 \lim_{y \to 0} y^a \partial_y w(\cdot,y) = -c_a (-\DD)^s f,
\end{equation}
where $c_a > 0$ is some explicit constant; see Proposition \ref{isom} below.

To give a precise meaning to the statements mentioned above, we first recall the definition of the homogeneous Sobolev spaces $\dot H^s(\R)$ as the completion of $C^\infty_0(\R)$ with respect to the quadratic form $\|(-\Delta)^{s/2} f\|^2$. It follows from Hardy's inequality that this completion is a space of functions when $0<s<1/2$. On the other hand, if $1/2\leq s\leq 1$, this completion is not a space of functions but rather a space of equivalence classes of functions differing by an additive constant. (To see this phenomenom for $s=1$, consider a smoothened version of the sequence $f_n(x)=(1-|x|/n)_+$. Similar examples can be constructed for any $1/2\leq s < 1$.) For simplicity, we shall write elements of $\dot H^s(\R)$ still as functions, but with the understanding that for $s\geq 1/2$ equalities are understood modulo constants.

Next, for $-1 < a < 1$ given, we introduce the weighted homogeneous Sobolev space $\dot{\H}^{1,a}(\R^2_+)$ as the completion of $C^\infty_0(\overline{\R^2_+})$ with respect to the quadratic form
\begin{equation}
\| u \|_{\dot{\H}^{1,a}(\R^2_+)}^2 := \iint_{\R^2_+} |\nabla u|^2 y^{a} \, dx \, dy .
\end{equation}
Similarly as before, this completion is a space of functions for $0<a<1$ and a space of equivalence classes modulo constants for $-1<a\leq 0$. (These facts are known, but they are also consequences of our analysis below.) We note that if $a=1-2s$, then $0<a<1$ if and only if $0<s<1/2$. Moreover, by scaling, one sees that $\int_\R y^{-1} P_a(x/y) \,dx$ is a constant independent of $y$ (indeed, it is $1$, as we shall see below). Hence if $f$ is an equivalence class modulo constants, then so is $\E_a f$. We have the following basic result.

\begin{prop}\label{isom}
Let $0 < s < 1$, $f \in \dot{H}^{s}(\R)$, and define $a  =1 -2s$. Then $\E_a f \in \dot{\H}^{1,a}(\R^{2}_+)$ and we have that
\begin{equation}\label{eq:extension}
\iint_{\R^2_+} |\nabla \E_a f|^2 y^a \,dx\,dy
= c_a \| (-\DD)^{\frac{s}{2}} f \|^2_2,
\end{equation}
where
\begin{equation}
 \label{eq:calpha}
c_a := 2^a \frac{\Gamma((1+a)/2)}{\Gamma((1-a)/2)} \,.
\end{equation}
Moreover, the function $w = \E_a f$ is a weak solution to the partial differential equation
\begin{equation}
 \label{eq:eq}
\mathrm{div}(y^a \nabla w) = 0 \quad \mbox{in $\R^2_+$} \, .
\end{equation}
Finally, we have $w(\cdot,\epsilon) \to f$ in $\dot H^s(\R)$ and $-c_a^{-1} \epsilon^a \frac{\partial w}{\partial y}(\cdot,\epsilon) \to (-\Delta)^s f$ in $\dot H^{-s}(\R)$, both as $\epsilon\to 0$.
\end{prop}

\begin{proof}
We begin by writing
\begin{align*}
\|(-\Delta)^{\frac{s}{2}} f \|^2_2 = \int \overline{(-\Delta)^s f(x)} f(x)\,dx ,
\end{align*}
where the right side should be understood as the duality pairing between $\dot H^{-s}$ and $\dot H^s$. Our goal now is to express both functions (which are strictly speaking distributions) on the right side as boundary values of functions defined on the upper half-plane $\R^2_+$. We put $w=\E_a f$ and claim that
\begin{equation}
 \label{eq:bv}
w(\cdot,\epsilon) \to f \quad\text{in}\ \dot H^s(\R)
\quad\text{and}\quad
-c_a^{-1} \epsilon^a \frac{\partial w}{\partial y}(\cdot,\epsilon) \to (-\Delta)^s f \quad\text{in}\ \dot H^{-s}(\R)
\end{equation}
both as $\epsilon\to 0$.

These properties are easily seen in Fourier space. Indeed, using the computation \cite[11.4.44]{AbSt64} of the Fourier transform of $P_a$, we see that
\begin{equation}\label{eq:e}
(\E_a f)(x,y) = (2\pi)^{-1/2} \int \hat f(\xi) m_a(|\xi| y) e^{i\xi\cdot x}\,d\xi \,,
\end{equation}
where
$$
m_a(r) := \frac2{\Gamma((1-a)/2)} \left(\frac r2\right)^{(1-a)/2} K_{(1-a)/2}(r)
$$
and $K_{(1-a)/2}$ is a Bessel function of the third kind. From standard properties of these functions (see, e.g., \cite{AbSt64} again) we know that $m_a(0)=1$ and $0< m_a(r) \leq A_a$ for all $r\geq 0$. This means that
$$
m_a(|\xi| \epsilon) \to 1
\quad\text{as}\ \epsilon\to 0 \quad\text{and}\quad
0< m_a(|\xi| \epsilon) \leq A_a \,,
$$
and hence
$$
\int_{\R} |\xi|^{2s} \left| m_a(|\xi| \epsilon) - 1 \right|^2 |\hat f(\xi)|^2 \,d\xi \to 0 \quad \mbox{as} \quad \eps \to 0,
$$
by dominated convergence. This proves the first relation in \eqref{eq:bv}. In order to prove the second one, we note that 
$$
\frac{\partial w}{\partial y}(x,\epsilon) = (2\pi)^{-1/2} \int \hat f(\xi) |\xi| m_a'(|\xi| \epsilon) e^{i\xi\cdot x}\,d\xi.
$$
and that, again by properties of Bessel functions, $\lim_{r\to 0} r^{a}m_a'(r) = -c_a$ and $0< -r^a m_a(r) \leq B_a$ for all $r\geq 0$. This means that
\begin{equation*}
 \label{eq:ptwconv}
-c_a^{-1} \epsilon^a |\xi| m_a'(|\xi| \epsilon) \to |\xi|^{2s}
\quad\text{as}\ \epsilon\to 0,
\quad
0<c_a^{-1} \epsilon^a |\xi| m_a'(|\xi| \epsilon) \leq B_a |\xi|^{2s} \,,
\end{equation*}
which, again by dominated convergence, implies that
$$
\int_{\R} \left| c_a^{-1} \epsilon^{a} |\xi| m_a'(|\xi| \epsilon) + |\xi|^{2s} \right|^2 |\hat f(\xi)|^2 \frac{d\xi}{|\xi|^{2s}} \to 0  \quad \mbox{as} \quad \eps \to 0,
$$
and thus establishing the second relation in \eqref{eq:bv}.

Next, we prove that $w=\E_a f$ satisfies the partial differential equation \eqref{eq:eq}. This can either be shown directly by differentiating \eqref{eq:econv}, or using \eqref{eq:e} and a partial Fourier transform with respect to $x$. Indeed, the Bessel equation satisfied by $K_{(1-a)/2}$ is equivalent to $(r^a m_a')' = r^a m_a$, which is the same as \eqref{eq:eq} after Fourier transform and scaling.

With \eqref{eq:bv} and \eqref{eq:eq} at hand, it is now easy to show that \eqref{eq:extension} holds. Indeed, 
\begin{align*}
\|(-\Delta)^{\frac{s}{2}} f \|^2_2
& = \int \overline{(-\Delta)^{s} f(x)} f(x)\,dx  = -c_a^{-1} \lim_{\epsilon\to 0} \epsilon^a 
\int \overline{\frac{\partial w}{\partial y}(x,\epsilon)} w(x,\epsilon)\,dx \\
& = c_a^{-1} \lim_{\epsilon\to 0} 
\iint_{\{y>\epsilon\}} |\nabla w(x,y)|^2 y^a \,dx\,dy.
\end{align*}
This proves that $\E_a f$ belongs to $\mathcal H^{1,a}(\R^{2}_+)$ and satisfies \eqref{eq:extension}. The proof of Proposition \ref{isom} is now complete. \end{proof}

For $u\in C_0^\infty(\overline{\R_+^{2}})$, we denote by $T u(x) := u(x,0)$ its {\em trace.} As we shall see, the operator $T$ can be extended by continuity to  ${\dot{\H}}^{1,a}(\R^2_+)$, thanks to the next proposition which also yields a sharp trace inequality. In particular, this auxiliary result identifies the space of functions on $\R$ that arise as traces of functions in $\dot{\H}^{1,a}(\R^{2}_+)$ as the homogeneous Sobolev space $\dot H^{s}(\R)$ when $s= \frac{1-a}{2}$.

\begin{prop} \label{lem:trace}
Let $0 < s <1$ and $a = 1-2s$. Then there is a unique linear bounded operator $T : \dot{\H}^{1,a}(\R^2_+) \to \dot{H}^{s}(\R)$ such that $Tu(x) = u(x,0)$ for $u \in C^\infty_0(\overline{\R^{2}_+})$. Moreover, for any $u \in \dot{\H}^{1,a}(\R^2_+)$, the following inequality holds
\begin{equation}\label{eq:restriction}
\iint_{\R^2_+} |\nabla u|^2 y^a \,dx\,dy
\geq c_a \| (-\DD)^{\frac{s}{2}} T u\|^2_2 ,
\end{equation}
with the constant $c_a$ from \eqref{eq:calpha}. Here equality is attained if and only if $u=\E_a f$ for some $f\in \dot H^{s}(\R)$.
\end{prop}

\begin{remark}
{\em In \cite{Go09} inequality \eqref{eq:restriction} was derived by different arguments in the range $1/2 < s < 1$.}
\end{remark}

\begin{proof}
We use a similar argument as in the proof of Proposition \ref{isom}. Let $u\in C_0^\infty(\overline{\R_+^2})$ and $g\in \dot H^{-s}(\R)$ be given. Note that $f:=(-\Delta)^{-s} g\in \dot H^s(\R)$. By the same arguments as in the proof of Proposition \ref{isom}, the function $w:=\E_a f$ satisfies \eqref{eq:eq} and \eqref{eq:bv}. Hence we conclude
\begin{align*}
\int \overline{g(x)} u(x,0)\,dx
& = \int \overline{(-\Delta)^s f(x)} u(x,0)\,dx \\ & = -c_a^{-1} \lim_{\epsilon\to 0} \epsilon^a 
\int \overline{\frac{\partial w}{\partial y}(x,\epsilon)} u(x,\epsilon)\,dx \\
& = c_a^{-1} \lim_{\epsilon\to 0} 
\iint_{\{y>\epsilon\}} \overline{\nabla w(x,y)} \cdot \nabla u(x,y) y^a \,dx\,dy \,.
\end{align*}
Next, by the Cauchy-Schwarz inequality,
\begin{align}\label{eq:csu}
\left|\int \overline{g(x)} u(x,0)\,dx \right|
& \leq c_a^{-1} \left( \iint |\nabla w(x,y)|^2 y^a \,dx\,dy \right)^{1/2} \\
& \qquad 
\left( \iint |\nabla u(x,y)|^2 y^a \,dx\,dy \right)^{1/2} \nonumber \,.
\end{align}
We also note that, by Proposition \ref{isom},
\begin{align*}
\limsup_{\epsilon\to 0} \iint_{\{y>\epsilon\}} |\nabla w(x,y)|^2 y^a \,dx\,dy
 = c_a \| (-\Delta)^{\frac{s}{2}} f \|^2_2
 = c_a \|(-\Delta)^{-\frac{s}{2}} g \|^2_2 \,.
\end{align*}
Thus we have shown that
\begin{align*}
\left| \int \overline{g(x)} u(x,0)\,dx \right|
& \leq c_a^{-1/2} \|(-\Delta)^{-\frac{s}{2}} g\|_2
\left( \iint |\nabla u(x,y)|^2 y^a \,dx\,dy \right)^{1/2} \,,
\end{align*}
which, by duality, is the same as \eqref{eq:restriction} for $u\in C_0^\infty(\overline{\R^2_+})$. This allows us to extend the operator $T$ by continuity from $C_0^\infty(\overline{\R_+^{2}})$ to $\dot{\H}^{1,a}(\R^{2}_+)$, preserving the above inequality, whereas the uniqueness of $T$ follows from the density of $C_0^\infty(\overline{\R_+^{2}})$.

Moreover, the above argument is valid for any $u\in \dot{\H}^{1,a}(\R^{2}_+)$ and equality in \eqref{eq:csu} is attained if and only if $\nabla u$ is a constant multiple of $\nabla w$. Hence $u$ is a weak solution of equation \eqref{eq:eq}. By the unique solvability of this equation, $u$ is given as the $\E_a$-extension of its trace.
\end{proof}

For the rest of this section, we will adapt the following convention. 

\begin{convention}
For $u \in \dot{\H}^{1,a}(\R^2_+)$, we also write $u(x,0)$ to denote its trace $Tu(x)$.
\end{convention}

We conclude our preliminary discussion by introducing the `inhomogeneous' Sobolev space 
\begin{equation}
\H^{1,a}(\R^2_+) := \{ u \in \dot{\H}^{1,a}(\R^2_+) : u(x,0) \in L^2(\R) \},
\end{equation}
endowed with the norm $\| u \|_{\H^{1,a}(\R^2_+)} := \| u \|_{\dot{\H}^{1,a}(\R^2_+)} + \| T u \|_{L^2(\R)}$. Note that $\H^{1,a}(\R^2_+)$ is in fact a space of functions (even for $-1<a\leq 0$). This space will be of use in the next subsection. 

\subsection{Variational Characterization of Eigenvalues}

Using the results of the previous subsection, we now derive a variational principle for the first $n \geq 1$ eigenvalues of a fractional Schr\"odinger operator $H=(-\Delta)^s +V$ in terms of a {\em local energy functional}.  Apart from requiring that $V$ be in the class $K_s$,  we make the convenient assumption that the essential spectrum of $H = (-\DD)^s + V$ satisfies
\begin{equation*}
\boxed{\sigma_{\mathrm{ess}}(H) = [0,\infty) } 
\end{equation*}
This can be imposed without loss of generality, by replacing $V$ with $V + c$ where $c \in \R$ is some suitable constant.

\medskip
We are now ready to formulate our key variational principle for the eigenvalues of $H$ below the essential spectrum.
 
\begin{corollary}\label{varprinc}
Let $0<s<1$ and $V \in K_s$. Suppose that $n \geq 1$ is an integer and assume that $H=(-\Delta)^s +V$ has at least $n$ eigenvalues 
 $$\lambda_1\leq \ldots\leq\lambda_n < 0 .$$
Furthermore, let $M$ be an $(n-1)$-dimensional subspace of $L^2(\R)$ spanned by eigenfunctions corresponding to the eigenvalues $\lambda_j$ with $j =1,\ldots,n-1$. Then we have
\begin{align*}
 \lambda_n = \inf\left\{ c_a^{-1} \iint_{\R^2_+} |\nabla u|^2 y^a \,dx\,dy + \int_{\R} V(x) |u(x,0)|^2 \,dx : \ u\in \mathcal H^{1,a}(\R^{2}_+),  \right. \\
 \left. \int |u(x,0)|^2 \,dx =1 , \, u(\cdot,0) \perp M \right\} \,,
\end{align*}
where $a = 1-2s$ and $c_a > 0$ is the constant from \eqref{eq:calpha}. Moreover, the infimum is attained if and only if $u=\E_a f$, where $\| f\|_2^2 =1$ and $f\in M^\bot$ is a linear combination of eigenfunctions corresponding to the eigenvalue $\lambda_n$.
\end{corollary}

\begin{proof}
By Lemma \ref{lem:trace}, the infimum on the right-hand side is bounded from below by
$$
\inf\left\{ \| (-\DD)^{\frac{s}{2}} f \|_2^2 + \int V |f|^2 \, dx :  f\in H^s(\R),\; \|f\|_2^2=1, \; f \perp M \right\} ,
$$
and equality is attained if and only if $u=\E_a f$. The assertion now follows from the usual variational characterization for the eigenvalues of $H=(-\DD)^s +V$.
\end{proof}


\subsection{Nodal Domain Bound on $\R^2_+$}

Let $V \in K_s$. Recall that we can always assume that any $L^2$-eigenfunction $\psi$ of $H$ is real-valued, since $H=(-\DD)^s + V$ is a real operator. Furthermore, by the  remark following Definition \ref{def:kato}, any such eigenfunction $\psi$ of $H$ is bounded and continuous. Likewise, the extension $\E_a \psi$ belongs to $C^0(\overline{\R^2_+})$ as well. Consider the set $N = \{ (x,y) \in \R^2_+ : (\E_a \psi)(x,y) = 0 \}$ which is a closed in $\R^2_+$. We define the {\em nodal domains} of $\E_a \psi$ as the connected components of the open set $\R^2_+ \setminus N$ in $\R^{2}_+$. We have the following bound on the number of nodal domains.
\begin{thm}\label{nodal}
Let $0<s<1$, $V \in K_s$, and define $a = 1-2s$. Suppose that $n \geq 1$ is an integer and assume that $H=(-\Delta)^s +V$ has at least $n$ eigenvalues 
 $$
 \lambda_1 \leq \ldots \leq \lambda_n < 0. $$
If $\psi_n \in H^s(\R) \cap C^0(\R)$ is a real eigenfunction of $H$ with eigenvalue $\lambda_n$, then its extension $\E_a \psi_n$ has at most $n$ nodal domains in $\R^2_+$.
\end{thm}

\begin{proof}
We argue by contradiction. Assume $\E_a \psi_n$ has nodal domains $\Omega_1,\ldots,\Omega_m$ where $m\geq n+1$. We consider the sets $K_j:=\{x\in\R: (x,0)\in \overline\Omega_j\}$ for $j=1,\ldots,m$, where $\overline \Omega_j$ is the closure of $\Omega_j$ in $\overline{\R_+^2}$. Since $\E_a \psi_n$ is continuous up to the boundary and $\psi_n\not\equiv 0$, we may assume that $K_1\neq\emptyset$. Furthermore, let $M$ be an $n-1$ dimensional subspace of $L^2(\R)$ spanned by eigenfunctions corresponding to the eigenvalues $\lambda_j$ where $j=1,\ldots,n-1$. Next, we define the function
$$
u = \left( \E_a\psi_n \right) \sum_{j=1}^{n} \gamma_j \1_{\overline{\Omega}_j} .
$$
Note that we can choose the constants $\gamma_j \in \R$, with $j=1,\ldots,n$, in such a way that $u(\cdot, 0) \perp  M$ and $\|u(\cdot,0)\|_2=1$. Using standard facts about Sobolev spaces one can show that $u \in \H^{1,a}(\R^2_+)$ and $\nabla u= \left( \nabla\E_a\psi_n \right) \sum_{j=1}^{n} \gamma_j \1_{\overline{\Omega}_j}$. 

By the same argument as in the proof of Proposition \ref{lem:trace}, the function $\E_a\psi_n$ satisfies
\begin{equation*}
c_a^{-1} \iint_{\R^2_+} \overline{\nabla v}\cdot (\nabla \E_a\psi_n) y^a \,dx\,dy + \int_\R V(x) \overline{v(x,0)} \psi_n(x) \,dx = \lambda_n  \int_\R \overline{v(x,0)} \psi_n(x) \,dx
\end{equation*}
for any $v\in \H^{1,a}(\R^2_+)$. Since $u$ belongs to $\H^{1,a}(\R^2_+)$, we can apply this to $v=u$ and obtain
\begin{align*}
& c_a^{-1} \iint_{\R^2_+} |\nabla u|^2 y^a \,dx\,dy + \int_{\R} V(x) |u(x,0)|^2 \,dx \\
& \quad = \lambda_n \sum_{j=1}^n |\gamma_j|^2 \int_{K_j} |\psi_n|^2 \,dx
= \lambda_n \| u(\cdot,0)\|^2_2 = \lambda_n \,.
\end{align*}
Thus we conclude that equality holds in the variational principle in Corollary~\ref{varprinc}. Hence $u=\E_a f$ where $f\in M^\bot$ is a linear combination of eigenfunctions corresponding to the eigenvalue $\lambda_n$. In particular, the non-trivial function $u$ satisfies equation \eqref{eq:eq}. Note that we have $u\equiv 0$ on the open non-empty set $\Omega_{n+1} \subset \R^2_+$. However, we can deduce by unique continuation of solutions for the elliptic equation \eqref{eq:eq} that $u \equiv 0$ on the upper halfplane $\R^2_+$. Indeed, consider the open connected set $D_{\delta} = \{ (x,y) \in \R^2_{+} : \delta  < y < 1/\delta \}$ where $0 <  \delta < 1$ is a fixed constant. Clearly, the differential operator $L$ on $D_\delta$ with $Lu = \mathrm{div} \, (y^a \nabla u)$ has smooth coefficients, and moreover $L$ is uniformly elliptic on $D_{\delta}$ (with bounds depending on $\delta$). By choosing $\delta_0 > 0$ now sufficiently small such that $\Omega_{n+1} \cap D_{\delta_0} \neq \emptyset$, we deduce by standard unique continuation for $L u=0$ that $u \equiv 0$ on the connected set $D_{\delta_0}$. We can repeat this argument for any set $D_{\delta} \subset \R^2_+$ with $0 < \delta \leq \delta_0$ to conclude that $u \equiv 0$ on $\R^2_+$ itself. But this is a contradiction. The proof of Theorem \ref{nodal} is now complete. \end{proof}

\subsection{Proof of Theorem \ref{thm:courant}}

We argue by contradiction. Suppose that $\psi_2 : \R \to \R$ changes its sign at least three times on $\R$. Thus, after replacing $\psi_2$ by $-\psi_2$ if necessary, there exist points $x_1 < x_2 < x_3 < x_4$ on the real line such that 
$$\mbox{$\psi(x_i) > 0$ for $i=1,3$}, \quad \mbox{$\psi(x_i) < 0$ for $i=2,4$.}$$

Next, we consider the extension $\E_a \psi_2$ on $\R^2_+$. Since $(\E_a \psi_2)(x_1, 0) > 0$ and $(\E_a \psi_2)(x_2, 0) < 0$ and by continuity of $\E_a \psi_2$ up to the boundary $\partial \R_+^2$, the function $\E_a \psi_2$ has at least two nodal domains in $\R_+^2$. But, in view of Theorem \ref{nodal}, we conclude that $\E_a \psi_2$ has exactly two nodal domains in $\R^2_+$, which we denote by $\Omega_+$ and $\Omega_-$ in the following.

Now, by continuity of $\E_{a}\psi_2(x,y)$ again, we deduce that
$$
\mbox{$(x_i,\eps) \in \Omega_+$ for $i=1,3$} \quad \mbox{and} \quad \mbox{$(x_i, \eps) \in \Omega_-$ for $i=2,4$,}
$$
for all $0 < \eps  \leq \eps_0$, where $\eps_0 > 0$ is some sufficiently small constant. Note that the connected open sets $\Omega_{\pm} \subset \R^2_+$ must be arcwise connected. Thus we conclude that there exist two simple continuous curves $\gamma_+, \gamma_- \in C^0([0,1]; \overline{\R^2_+})$ with the following properties.
\begin{itemize}
\item $\gamma_+(0) = (x_1,0), \gamma_+(1) = (x_3,0)$ and $\gamma_+(t) \in \Omega_+$ for $t \in (0,1)$.
\item $\gamma_-(0) = (x_2,0), \gamma_-(1) = (x_4,0)$ and $\gamma_-(t) \in \Omega_-$ for $t \in (0,1)$.
\end{itemize}
By Lemma \ref{lem:jordan} (based on a basic topological arguments) we deduce that $\gamma_+$ and $\gamma_-$ must intersect in the upper halfplane $\R^2_+$. But this contradicts $\Omega_+ \cap \Omega_- = \emptyset$.  Hence the function $\psi_2 : \R \to \R$ changes its sign at most twice on $\R$.

Finally, if $\psi_2 = \psi_2(|x|)$ is even, then clearly $\psi_2$ can change its sign on $\{ x > 0 \}$ at most once, since otherwise $\psi_2$ would change it sign at least four times on $\R$, contradicting the result just derived. By the remark following Theorem \ref{thm:courant}, we deduce that $\psi_2$ must change its sign at least once on $\{x > 0 \}$.  This completes the proof of Theorem \ref{thm:courant}.  \hfill $\blacksquare$


\section{Nondegeneracy of Ground States}

\label{sec:nondeg}

This section is devoted to the proof of Theorem \ref{thm:nondeg}. That is, we show that (local) nonnegative minimizers $Q(x) \geq 0$ for the functional $J^{s,\alpha}(u)$ defined \eqref{eq:Jdef} have a nondegenerate linearization. In fact, we shall prove a slightly more general result formulated as Lemma \ref{lem:kernel} below.  

Let $0 < s < 1$ and $0 < \alpha < \amax(s)$ be fixed throughout this section. Suppose that $Q \in H^s(\R)$ with $Q \not \equiv 0$ is a nonnegative solution $Q =Q(x) \geq 0$ of
\begin{equation} \label{eq:phiLp}
(-\DD)^s Q + \lambda Q - Q^{\alpha+1}  = 0 ,
\end{equation}
with some positive constant $\lambda > 0$. Note that, by Lemma \ref{lem:symm} which is based on the method of moving planes, we have that $Q(x) = \widetilde{Q}(|x-x_0|) > 0$ for some $x_0 \in \R$, where $\widetilde{Q} = \widetilde{Q}(|x|) > 0$ is an even and positive function strictly decreasing in $|x|$. Moreover, a simple rescaling argument shows that we could assume $\lambda = 1$ without loss of generality. But for the sake of later purpose, we will keep $\lambda > 0$ explicit here. 

Associated with $Q \in H^s(\R)$, we define the linearized operator
\begin{equation} \label{eq:defLL}
L_+ = (-\DD)^s + \lambda - (\alpha+1) Q^\alpha
\end{equation}
acting on $L^2(\R)$. We record the following (partly immediate) facts about $L_+$.
\begin{itemize}
\item $Q^\alpha \in K_s$, i.\,e., the potential $V=Q^\alpha$ belongs to the `Kato-class' with respect to $(-\DD)^s$. This follows from the remark following Definition \ref{def:kato} and Sobolev inequalities. In particular, any $L^2$-eigenfunction of $L_+$ is continuous and bounded.
\item $L_+$ is a self-adjoint operator on $L^2(\R)$ with quadratic form domain $H^{s}(\R)$ and operator domain $H^{2s}(\R)$.
\item The essential spectrum is $\sigma_{\mathrm{ess}}(L_+) = [\lambda, \infty)$.
\item The Morse index of $L_+$ satisfies $\cN_-(L_+) \geq 1$. Recall that $\cN_-(L_+)$ is defined as the number of strictly negative eigenvalues, i.\,e.,
$$
\cN_- (L_+) = \# \{ e < 0 : \mbox{$e$ is eigenvalue of $L_+$ acting on $L^2(\R)$} \},
$$
where multiplicities of eigenvalues are taken into account. To see that indeed $\cN_-(L_+)\geq 1$, we just use $\langle Q, L_+ Q \rangle = -\alpha \| Q \|_{\alpha+2}^{\alpha+2} < 0$ by \eqref{eq:phiLp}. Thus, by the min-max principle, the operator $L_+$ has at least one negative eigenvalue.

\item We always have that $L_+ Q'=0$ and thus $\mathrm{span}\, \{ Q' \} \subseteq \mathrm{ker} \, L_+$. This follows from differentiating \eqref{eq:phiLp} with respect to $x$.
\item The lowest eigenvalue $e_0 = \inf \sigma \, (L_+)$ is simple and the corresponding eigenfunction $\psi_0=\psi_0(x) > 0$ strictly positive; see Lemma \ref{lem:perron}. 
\end{itemize}

\noindent

To formulate the main result of this section, we now suppose that $Q = Q(|x|)$ is an even function. We introduce the {\em Morse index of $L_+$ in the sector of even functions} by defining
$$
\boxed{
\cN_{-,\mathrm{even}} (L_+) := \# \{ e < 0 : \mbox{$e$ is eigenvalue of $L_+$ restricted on $L^2_{\mathrm{even}}(\R)$} \} }
$$
where multiplicities of eigenvalues are taken into account. Note that $\langle Q, L_+ Q \rangle < 0$ with $Q=Q(|x|)$ even. Hence we deduce the general lower bound $\cN_{-, \mathrm{even}}(L_+) \geq 1$ from the min-max principle.

The key nondegeneracy result of this section is now as follows.

\begin{lemma} \label{lem:kernel}
Let $Q \in H^s(\R)$ be an even and positive solution of \eqref{eq:phiLp} with some $\lambda > 0$. Consider its associated linearized operator $L_+$ acting on $L^2(\R)$ and assume its Morse index in the even sector satisfies $\cN_{-,\mathrm{even}} (L_+) = 1$. Then we have
$$\mathrm{ker} \, L_+ = \mathrm{span} \, \{ Q ' \}.$$
\end{lemma}

\begin{proof}
By rescaling $Q(x) \mapsto \lambda^{\frac{1}{\alpha}} Q(\lambda^\frac{1}{2s}x)$ (and likewise any element in  $\mathrm{ker} \, L_+$ transforms accordingly), we can assume that $\lambda = 1$ holds in \eqref{eq:phiLp}. 

Next, we consider the orthogonal decomposition $L^2(\R) = L^2_{\mathrm{even}}(\R) \oplus L^2_{\mathrm{odd}}(\R)$. Since $Q = Q(|x|)$ is an even function, we note that $L_+$ leaves the subspaces $L^2_{\mathrm{even}}(\R)$ and $L^2_{\mathrm{odd}}(\R)$ invariant. We treat these subspaces separately as follows.

Recall that $Q' \in L^2_{\mathrm{odd}}(\R)$ satisfies $L_+ Q' = 0$. Moreover, by Lemma \ref{lem:symm}, we have that $Q'(x) < 0$ for $x > 0$.  In view of Lemma \ref{lem:perron2} applied to $L_+$, we conclude that $Q'$ is (up to a sign) the unique ground state eigenfunction of $L_+$ restricted to $L^2_{\mathrm{odd}}(\R)$. Hence we see that $ \mathrm{ker} \, L_+ |_{L^2_\mathrm{odd}}=\mathrm{span} \, \{ Q ' \} .$

It remains to show that 
$$\mathrm{ker} \, L_+ |_{L^2_\mathrm{even}} = \{ 0\}.$$ 
To prove this claim, we argue by contradiction. Suppose there exists $v \in L^2_{\mathrm{even}}(\R)$ with $v \not \equiv 0$ such that $L_+ v = 0$. Note that $v$ is continuous and bounded due to the remarks above. Also, since $L_+$ is a real operator, we can assume that $v$ is real-valued. Next, by assumption, we have $\cN_{-, \mathrm{even}}  (L_+) = 1$, and hence $v$ must be an eigenfunction of $L_+$ corresponding to its second eigenvalue. Let $\psi_1$ be the first eigenfunction to $L_+$. By Lemma \ref{lem:perron2}, we have that $\psi_1(x) > 0$ is strictly positive. By the orthogonality $v \perp \psi_1$, we deduce that $v$ must change its sign at least once on $\R$. Moreover since $v$ is even, this implies that $v$ changes its sign at least twice on $\R$. But, by applying Theorem \ref{thm:courant} to $H= (-\DD)^s - (\alpha+1) Q^\alpha$, we deduce that $v$ changes its sign exactly twice on $\R$. Since $v = v(|x|)$ is even, this implies that there exists $r_* > 0$ such that the following holds (after multiplying $v$ by $-1$ if necessary):
\begin{equation} \label{eq:vnodal}
\mbox{$v(|x|) \geq 0$ for $|x| \leq r_*$}, \quad \mbox{$v(|x|) \leq 0$ for $|x| > r_*$}, 
\end{equation}
where $v \not \equiv 0$ on both sets $\{ |x| \leq r_* \}$ and $\{ |x| > r_*\}$. Note that we have the same estimate on the number of sign changes of $v=v(|x|)$ on the halfline $(0,\infty)$, as if Sturm-Liouville oscillation theory for ODE were applicable. Therefore we can now proceed along the lines of \cite{ChGuNaTs07}, where a simple nondegeneracy proof for NLS ground states was given based on a result from Sturm-Liouville theory. Adapting this argument to our setting, we notice that a calculation yields
\begin{equation}
L_+ Q = -\alpha Q^{\alpha+1} \quad \mbox{and} \quad L_+ R = -2sQ,
\label{eq:Gust}
\end{equation}
where
\begin{equation}
R := \frac{d}{d \beta} \Big |_{\beta=1} \beta^{\frac{2s}{\alpha}} Q(\beta \cdot)  = \frac{2s}{\alpha} Q + x Q'.
\end{equation}
Note that $R \in L^2(\R)$ due to the decay estimate for $Q$ stated in Proposition \ref{prop:Q}. By bootstrapping the equation satisfies by $L_+$, we further deduce that $R \in H^{2s+1}(\R)$ and, in particular, we see that $R$ is in the domain of $L_+$. Since $L_+$ is self-adjoint and $v \in \mathrm{ker} \, L_+$, we obtain from \eqref{eq:Gust} that 
$$\langle Q^{\alpha+1}, v \rangle = \langle Q,v \rangle = 0.$$
Next, we consider the even function $f \in \mathrm{ran} \, L_+$ given by
$$
f := Q^{\alpha+1} - \mu Q = Q (Q^{\alpha} - \mu) ,
$$
where $\mu \in \R$ is some parameter. Note that $\langle v,f \rangle = 0$ for all $\mu \in \R$. Now we choose $\mu = ( Q(r_*))^\alpha$ with $r_* >0$ from \eqref{eq:vnodal}. Since $Q=Q(|x|) > 0$ is positive and strictly decreasing in $|x|$, we deduce that
\begin{equation} \label{eq:fprop}
\mbox{$f(|x|) >0$ for $|x| < r_*$},  \quad \mbox{$f(|x|) < 0$ for $|x| > r_*$}.
\end{equation}
Combining now \eqref{eq:fprop} and \eqref{eq:vnodal}, we see that $v f \geq 0$ with $v f\not \equiv 0$. Hence $\langle v,f \rangle > 0$. But this violates the orthogonality condition $\langle v,f\rangle =0$. Therefore, the operator $L_+$ does not have a zero eigenfunction in $L^2_{\mathrm{even}}(\R)$. The proof of Lemma \ref{lem:kernel} is now complete. \end{proof}

\subsection*{Proof of Theorem \ref{thm:nondeg}}
Suppose that $Q=Q(x) > 0$ is a positive solution to \eqref{eq:phiLp} with $\lambda=1$. By Lemma \ref{lem:symm} and translational invariance, we can assume that $Q=Q(|x|) > 0$ is even. 

Let $L_+ = (-\DD)^s + 1 -(\alpha+1) Q^\alpha$ be the associated linearized operator. To apply Lemma \ref{lem:kernel}, it suffices to show that $\cN_{-, \mathrm{even}} (L_+) = 1$ holds. Indeed, we recall that, by assumption in Theorem \ref{thm:nondeg}, the function $Q$ is a local minimizer of $J^{s,\alpha}(u)$. Therefore, we have the second order condition
 \begin{equation}
 \label{eq:stable}
\frac{d^2}{d \varepsilon^2}  \Big |_{\varepsilon =0} J^{s,\alpha}(Q + \varepsilon \eta ) \geq 0 \quad \mbox{for all $\eta \in C^\infty_0(\R)$.}
\end{equation}
We claim that this implies the upper bound $\cN_{-,\mathrm{even}}(L_+) \leq \cN_{-} (L_+) \leq 1$. To estimate the Morse index, we can adapt an argument for ground states of classical nonlinear Schr\"odinger equations (see \cite{ChGuNaTs07,We85}) to  our setting as follows.

By Lemma \ref{lem:pohoidentities} below (with $\lambda_s=1$ and $s=s_0$), we obtain the following ``Pohozaev identities'' of the form
$$
\frac{\int | (-\DD)^\frac{s}{2} Q|^2}{2a} =  \frac{\int |Q|^2}{ 2b } = \frac{ \int |Q|^{\alpha+2}}{\alpha+2} =: k,
$$
where $a = \frac{\alpha}{4s} > 0$, $b = \frac{\alpha}{4s}(2s-1) +1 > 0$, and $k > 0$ are positive constants. Then an elementary (but tedious) calculation shows  that inequality \eqref{eq:stable} is equivalent to
\begin{equation*}
k \langle \eta, L_+ \eta \rangle \geq \frac{1}{a} | \langle \eta, (-\DD)^s Q \rangle | ^2 + \frac{1}{ b}  \left| \langle \eta, Q\rangle \right|^2 -  \left| \langle \eta, Q^{\alpha+1}\rangle \right|^2.
\end{equation*}
Therefore $\langle \eta, L_+ \eta\rangle \geq 0$ if $\eta \perp Q^{\alpha+1}$. Hence, by the min-max principle, we obtain that $\cN_-  (L_+) \leq 1$ and hence $\cN_{-,\mathrm{even}}(L_+) \leq \cN_-  (L_+) \leq 1$.

On the other hand, we recall that we always have  that $\cN_{-,\mathrm{even}} (L_+) \geq 1$, as remarked in the beginning of this section. Thus we conclude that $\cN_{-,\mathrm{even}}  (L_+) = 1$ holds, whence it follows that $\mathrm{ker} \, L_+ = \mathrm{span} \, \{ Q' \}$  thanks to Lemma \ref{lem:kernel}. The proof of Theorem \ref{thm:nondeg} is now complete. \hfill $\blacksquare$


\section{Uniqueness of Ground States}

\label{sec:unique}

In this section we prove Theorem \ref{thm:unique}. Our strategy is based on the nondegeneracy result from Section \ref{sec:nondeg} and an implicit function argument, combined with a global continuation argument. For the reader's orientation, we first give a brief outline of this section as follows. 

In Subsection \ref{subsec:localbranch}, we fix $0 < s_0 < 1$ and $0 < \alpha < \amax(s_0)$. By an implicit function argument, we construct (in some suitable Banach space of even functions) a locally unique branch $(Q_s, \lambda_s)$ parameterized by $s$ close to $s_0$ and satisfying
$$
(-\DD)^s Q_s + \lambda_s Q_s - |Q_s|^{\alpha} Q_s = 0.
$$
Here the starting point of the branch $(Q_0,\lambda_0)=(Q_{s=s_0}, \lambda_{s=s_0})$ is assumed to satisfy some spectral condition; see Proposition \ref{prop:local}.

Then, in Subsection \ref{subsec:global}, we show (as a main result of this section) that the local branch $(Q_s, \lambda_s)$ can be indeed globaly continued to $s=1$, provided that $(Q_0, \lambda_0)$ satisfies some explicit conditions, such as positivity $Q_0 = Q_0(|x|) > 0$; see Proposition \ref{prop:global}. The crucial and delicate point that allows us to extend to $s=1$ is based on suitable a-priori bounds on regularity and spatial decay for $(Q_s,\lambda_s)$ of the form
$$
 1 \lesssim \int | (-\DD)^{\frac{s}{2}} Q_s|^2 \lesssim 1, \quad 1 \lesssim  \int |Q_s|^2 \lesssim 1, \quad 1 \lesssim \lambda_s \lesssim 1,
$$
in combination with a uniform pointwise decay estimate $Q_s(|x|) \lesssim |x|^{-1}$ for $|x| \gtrsim 1$. The derivation of all these bounds will cover most of this section and it requires a careful study of the nonlinear problem.

Finally, with help of Propositions \ref{prop:local} and \ref{prop:global}, we are able to prove Theorem \ref{thm:unique} in Subsection \ref{subsec:unique} below. That is, we show that the branch $(Q_s, \lambda_s)$ starting from a ground state $(Q_0, \lambda_0=1)$ exists and is globally unique; in particular, the assumption of having another branch starting from a different ground state $(\widetilde{Q}_0, \widetilde{\lambda}_0=1)$ leads to a contradiction. This will follow from the global uniqueness and nondegeneracy for the limiting problem when $s=1$, i.\,e.,
$$
-\Delta Q_* + \lambda_* Q_* - Q_*^{\alpha+1} = 0.
$$






\subsection{Construction of a Local Branch}

\label{subsec:localbranch}

We start with some preliminaries. Let $0 < s < 1$ and $0 < \alpha < \amax(s)$ be given. We consider solutions $(Q,\lambda)$ with $Q \in L^2(\R) \cap L^{\alpha+2}(\R)$ and $\lambda \in \Rplus$ satisfying 
\begin{equation} \label{eq:phis}
(-\DD)^s Q + \lambda Q - |Q|^{\alpha} Q = 0.
\end{equation}
In fact, by a bootstrap argument, we see that $Q \in H^{2s+1}(\R)$ holds. Nevertheless, it turns out to be convenient to work in the space $L^2(\R) \cap L^{\alpha+2}(\R)$, which is independent of $s$. Since we are interested in real-valued and even solutions only, it is convenient to define the (real) Banach space
\begin{equation} \label{def:X}
\boxed{ 
X^{\alpha} := \big \{ f \in L^2(\R) \cap L^{\alpha+2}(\R) : \mbox{$f$ is even and real-valued} \big \} }
\end{equation}
which we equip with the norm $$\| f \|_{X^{\alpha}} := \| f \|_2 + \| f \|_{\alpha+2}.$$  Recall that we make the standard abuse of notation by writing both $f(x)$ and $f(|x|)$ whenever $f$ is an even function on $\R$. 

As a next step, we will construct a local branch of solutions $(Q_s, \lambda_s) \in X^\alpha \times \Rplus$ of \eqref{eq:phis}, which is parametrized by $s$ in some small interval. To this end, we introduce the following assumption.

\begin{assumption} \label{ass:branch}
Let $0 < s < 1$, $0 < \alpha < \amax(s)$. Suppose that $(Q, \lambda) \in X^\alpha \times \Rplus$ satisfies equation \eqref{eq:phis}. We assume that the linearized operator 
$$L_+ = (-\DD)^{s} + \lambda - (\alpha+1) |Q|^{\alpha}$$ 
acting on $L^2(\R)$ has a bounded inverse $L_+^{-1}$ on $L^2_{\mathrm{even}}(\R)$.
\end{assumption}

\begin{remarks}
{\em 1.) We emphasize that we do not require $Q \in X^\alpha$ to be positive here.

2.) Since $Q \in X^\alpha$ is even and hence $Q \perp Q'$ in $L^2(\R)$, the bounded inverse $L_+^{-1}$ exists on $L^2_{\mathrm{even}}(\R)$, provided that $\ker L_+ = \mathrm{span} \{Q'\}$ holds for $L_+$ acting on $L^2(\R)$.

3.) By Sobolev inequalities, the invertibility of $L_+^{-1}$ on $L^2_{\mathrm{even}}(\R)$ implies that $L_+^{-1}$ exists on $X^\alpha$ as well.}
\end{remarks}

As a next step, we establish existence and local uniqueness of a branch $(Q_s, \lambda_s)$ for \eqref{eq:phis} around a solution $(Q_0, \lambda_0)$ that satisfies Assumption \ref{ass:branch}.
 
\begin{prop} \label{prop:local}
Let $0 < s_0 < 1$ and $0 < \alpha < \amax(s_0)$ be given. Suppose that $(Q_0, \lambda_0) \in X^\alpha \times \Rplus$ satisfies Assumption \ref{ass:branch} with $s=s_0$ and $\lambda=\lambda_0$. Then, for some $\delta > 0$, there exists a map $(Q, \lambda) \in C^1 \big ( I; X^\alpha \times \Rplus)$ defined on the interval $I = [s_0, s_0 +\delta)$ such that the following holds, where we denote $(Q_s, \lambda_s) = (Q(s), \lambda(s))$ in the sequel.
\begin{enumerate}
\item[(i)] $(Q_s, \lambda_s)$ solves equation \eqref{eq:phis} with $\lambda=\lambda_s$ for all $s \in I$.
\item[(ii)] There exists $\eps > 0$ such that $(Q_s , \lambda_s)$ is the unique solution of \eqref{eq:phis} for $s \in I$ in the neighborhood $\{ (Q,\lambda) \in X^\alpha \times \Rplus : \|  Q - Q_0 \|_X + |\lambda - \lambda_o | < \eps \}$. In particular, we have that $(Q_{s_0}, \lambda_{s_0}) = (Q_0, \lambda_0)$ holds.
\item[(iii)]  For all $s \in I$, we have
$$
\int | Q_s |^{\alpha+2} = \int | Q_{0} |^{\alpha+2}.
$$
\end{enumerate}

\end{prop}

\begin{remarks}{\em 
1.) Introducing the function $\lambda_s=\lambda(s)$ ensures that the above ``conservation law'' for $\int |Q_s|^{\alpha+2}$ holds. The use of this fact will become evident further below when we derive a-priori bounds.

2.) Note that $Q_s \in H^{2s+1}(\R)$ by standard regularity arguments; see Section \ref{app:regularity} below. But since the parameter $s$ changes, it is convenient to make use of the $s$-independent space $X^\alpha$ defined above.}
\end{remarks}

\begin{proof}
We use an implicit function argument as follows. First, we observe that \eqref{eq:phis} can be written as 
$$
Q =  \frac{1}{(-\DD)^s + \lambda} |Q|^{\alpha} Q .
$$
For some small constant $\delta > 0$ chosen below, we consider the mapping
\begin{equation} \label{def:F}
F : X^\alpha \times \Rplus \times [s_0, s_0 + \delta ) \to X^\alpha \times \R,
\end{equation}
which we define as
\begin{equation}
F(Q, \lambda, s) := \left [\begin{array}{l} \displaystyle Q - \frac{1}{(-\DD)^s + \lambda} |Q|^\alpha Q \\[1ex] \| Q \|_{{\alpha+2}}^{\alpha+2} - \| Q_{s_0} \|_{{\alpha+2}}^{\alpha+2} \end{array} \right ] .
\end{equation}
As shown in Lemma \ref{lem:F_C1}, the map $F$ is well-defined and $C^1$. Also, by construction, we have that $F(Q_{s_0}, \lambda_0, s_0) = 0$. To invoke an implicit function argument, we have to show the invertibility of the Fr\'echet deriviative of $F$ with respect $(Q,\lambda)$ at $(Q_0,\lambda_0,s_0)$, which we establish next.

First, we note that the Fr\'echet derivative of $F$ with respect to $(Q, \lambda)$ is given by
\begin{equation}
\partial_{(Q,\lambda)} F  = \left [ \begin{array}{cc}\displaystyle 1 - \frac{1}{(-\DD)^s+\lambda} (\alpha+1) |Q|^\alpha & \displaystyle  \frac{1}{( (-\DD)^s + \lambda )^2} |Q|^{\alpha} Q \\
(\alpha+2) \langle |Q|^\alpha Q, \cdot \rangle  & 0 \end{array} \right ],
\end{equation}
Here $\langle f, \cdot \rangle$ denotes the map $g \mapsto \langle f, g \rangle$. See also Lemma \ref{lem:F_C1} and its proof.

Now, we claim that the inverse $(\partial_{(Q,\lambda)} F)^{-1}$ exists at $(Q_{s_0}, \lambda_0, s_0)$. Hence we have to show that, for every $f \in X^\alpha$ and $\beta \in \R$ given, there is a unique solution $(\eta,\gamma) \in X^\alpha \times \R$ of the system
\begin{equation} \label{eq:system1}
(1+K) \eta + \gamma g = f ,
\end{equation}
\begin{equation} \label{eq:system2}
(\alpha+2) \langle |Q_{s_0}|^{\alpha} Q_{s_0}, \eta \rangle = \beta ,
\end{equation}
where we set
\begin{equation}
K :=  -\frac{1}{(-\DD)^{s_0} + \lambda_0} (\alpha+1) | Q_{s_0}| ^{\alpha}, \quad g :=  \frac{1}{((-\DD)^{s_0} + \lambda_0)^2} |Q_{s_0}|^{\alpha} Q_{s_0} .
\end{equation} 

Next, we note that $K$ is a compact operator on $L^2_{\mathrm{even}}(\R)$. Moreover, we see that $-1 \not \in \sigma(K)$ due to Assumption \ref{ass:branch}. Indeed, assume on the contrary that $-1$ is in the spectrum $\sigma(K)$ for $K$ acting on $L^2_{\mathrm{even}}(\R)$. Then the self-adjoint operator 
\begin{equation} \label{def:Limpl}
L_+ = (-\DD)^{s_0} + \lambda_0 -  (\alpha+1) |Q_{s_0}|^{\alpha}
\end{equation}
would have an even eigenfunction $v \in L^2_{\mathrm{even}}(\R)$ such that $L_+ v = 0$, which contradicts Assumption \ref{ass:branch}.

Thus the operator $(1+K)$ is invertible on $L^2_{\mathrm{even}}(\R)$. Moreover, since $K : X^\alpha \to X^\alpha$ holds (see proof of Lemma \ref{lem:F_C1} for details), we deduce that $(1+K)^{-1}$ exists on the space $X^\alpha$ as well. Hence we can solve \eqref{eq:system1} uniquely for $\eta \in X^\alpha$ to find that
\begin{equation}
\eta = (1+K)^{-1} f - \gamma (1+K)^{-1} g . 
\end{equation}
Plugging this into \eqref{eq:system2} yields
\begin{equation} \label{eq:gamma}
(\alpha+2) \langle | Q_{s_0} |^{\alpha} Q_{s_0}, (1+K)^{-1} g \rangle \gamma = - \beta + (\alpha+2) \langle |Q_{s_0}|^{\alpha} Q_{s_0}, (1+K)^{-1} f \rangle .
\end{equation}
To deduce unique solvability for $\gamma \in \R$, it remains to show that the coefficient in front of $\gamma$ does not vanish. To see this, we observe the identity
\begin{equation}
(1+K)^{-1} = L_+^{-1} \circ \big ((-\Delta)^{s_0} + \lambda_0 \big ) ,
\end{equation}
with $L_+$ given by \eqref{def:Limpl}. Using this identity together with $L_+ Q_{s_0} = -\alpha |Q_{s_0}|^{\alpha} Q_{s_0}$ and equation \eqref{eq:phis} satisfied by $Q_{s_0}$, we now easily deduce that
\begin{equation}
\big \langle |Q_{s_0}|^{\alpha} Q_{s_0}, (1+K)^{-1} g \big \rangle = -\frac{1}{\alpha} \int | Q_{s_0}Ê|^2 \neq 0,
\end{equation}
This completes the proof that $\partial_{(Q,\lambda)} F$ is invertible at $(Q_{s_0}, \lambda_0, s_0)$. By applying the implicit function theorem to the map $F$ at $(Q_{s_0}, \lambda_0, s_0)$, we derive the assertions (i)--(iii) provided that $\delta > 0$ is sufficiently small.
 
The proof of Proposition \ref{prop:local} is now complete.\end{proof}

\subsection{A-priori Bounds and Global Continuation} 

\label{subsec:global}

Let $0 < s_0 < 1$ and $0 < \alpha < \amax(s_0)$ be given. Throughout this subsection, we suppose that $(Q_s, \lambda_s) \in C^1(I; X^\alpha \times \Rplus)$ is a local branch defined for $I = [s_0, s_0 +\delta)$, as provided by Proposition \ref{prop:local}. 

Now, we consider the corresponding {\em maximal extension} of the branch $(Q_s, \lambda_s)$ for $s \in [s_0, s_*)$, where $s_*$ is given by
\begin{align*} \label{def:smax}
s_* := \sup \Big \{ s_0 < \tilde{s} < 1 & : \mbox{$(Q_s,\lambda_s) \in C^1([s_0,\tilde{s}); X^\alpha \times \Rplus)$ given by Proposition \ref{prop:local}} \nonumber \\
& \quad  \mbox{and $(Q_s,\lambda_s)$ satisfies Assumption \ref{ass:branch} for $s \in [s_0,\tilde{s})$} \Big \} .  
\end{align*}
Clearly, we have $s_* \leq 1$ and our goal will be to show that $s_* = 1$ holds under some suitable assumption on $(Q_{s_0}, \lambda_0)$. 

Since we need to derive suitable a-priori bounds for the maximal branch $(Q_s, \lambda_s)$, we introduce the convenient notation 
$$
\boxed{ a \lesssim b }
$$
whenever $a \leq Cb$ holds, where $C>0$ is some constant that only depends on the fixed quantities $s_0, \alpha$ and $(Q_0, \lambda_0)$. As usual, the constant $C > 0$ is allowed to change from inequality to inequality.

As an initial step to derive a-priori bounds, we start with the following {\em ``Pohozaev identities''} satisfied by $Q_s$.

\begin{lemma} \label{lem:pohoidentities}
For all $s \in [s_0, s_*)$, the following identities hold:
$$
\frac{\lambda_s}{2} \int |Q_s|^2 =  \frac{a_s}{\alpha+2} \int |Q_s|^{\alpha+2} , \quad 
 \frac{1}{2} \int |(-\DD)^\frac{s}{2} Q_s |^2 =   \frac{b_s}{\alpha+2} \int |Q_s|^{\alpha+2} ,
$$
where $a_s = \frac{\alpha}{4s}(2s-1)+1$ and $b_s = \frac{\alpha}{4s}$.
\end{lemma}

\begin{proof}
By integrating \eqref{eq:phis} against $Q_s \in H^s(\R)$, we obtain
\begin{equation} \label{eq:app_poho1}
\int | (-\DD)^\frac{s}{2} Q_s|^2 + \lambda_s \int |Q_s|^2 = \int |Q_s|^{\alpha+2}. 
\end{equation}
Furthermore, we integrate \eqref{eq:phis} against $x Q'_s$. Integrating by parts yields the identities $\langle x Q'_s, |Q_s|^\alpha Q_s \rangle = -\frac{1}{\alpha+2} \int |Q_s|^{\alpha+2}$ and $\langle x Q_s', (-\DD)^s Q_s \rangle = \frac{2s-1}{2} \langle Q_s, (-\DD)^s Q_s \rangle$, where the second one follows from $- [\nabla \cdot x, (-\DD)^s] = 2s (-\DD)^s$. Hence, we deduce
\begin{equation} \label{eq:app_poho2}
\frac{2s-1}{2} \int |(-\DD)^\frac{s}{2} Q_s |^2 - \frac{\lambda_s}{2}  \int |Q_s|^2 = -\frac{1}{\alpha+2} \int |Q_s|^{\alpha+2} .
\end{equation}
(Note that the calculations here involving $x Q'_s$ are well-defined, thanks to the regularity and decay estimates from Proposition \ref{prop:Q}.) By combining equations \eqref{eq:app_poho1} and \eqref{eq:app_poho2}, we readily deduce Lemma \ref{lem:pohoidentities}.\end{proof}

Next, we derive the following straightforward a-priori bounds.

\begin{lemma} \label{lem:pohobounds}
For all $s \in [s_0, s_*)$, we have the following bounds
$$
1 \lesssim  \int | (-\DD)^{\frac{s}{2}} Q_s |^2 \lesssim 1,   \quad
1 \lesssim \lambda_s \int |Q_s|^2  \lesssim  1, \quad
1 \lesssim \int | Q_s |^2.
$$
\end{lemma}

\begin{proof}
Using Lemma \ref{lem:pohoidentities}, we obtain the desired a-priori bounds for $\int |(-\DD)^{\frac{s}{2}} Q_s|^2$ and $\lambda_s \int |Q_s|^2$, since we have  $\int |Q_s|^{\alpha+2} = \mbox{const}. \neq 0$ along the branch $(Q_s, \lambda_s)$ and clearly $1 \lesssim a_s, b_s \lesssim 1$ holds for $s \in [s_0,s_*)$.

To derive the lower bound $\int |  Q_s |^2 \gtrsim 1$, we recall the interpolation estimate \eqref{ineq:GN}, which yields that
\begin{equation*}
\left ( \int | Q_s |^2 \right )^{\frac{\alpha}{4s}(2s-1) + 1} \geq \frac{1}{C_{s, \alpha}} \frac{ \int |Q_s|^{\alpha+2}  }{ \left ( \int | (-\DD)^{\frac{s}{2}} Q_s |^2 \right )^{\frac{\alpha}{4s}} } \geq \frac{1}{C} \, ,
\end{equation*}
with some  constant $C > 0$ independent of $s$. Here we used again that $\int |Q_s|^{\alpha+2} = \mbox{const}.$ and $\int | (-\Delta)^{\frac{s}{2}} Q_s | ^2 \lesssim 1$ from above, as well as the fact that the optimal interpolations constants satisfy $C_{s, \alpha} \leq K$ by Lemma \ref{lem:GNuniform} with some constant $K > 0$ uniformly in $s \geq s_0 > \frac{\alpha}{2 (\alpha+2)}$. Here the last strict inequality is due to $\alpha < \amax(s_0)$. \end{proof}

We now derive an a-priori {\em upper bound} for $\int | Q_s |^2$ along the branch $(\lambda_s, Q_s)$. In fact, this result will be one of the key steps in order to extend the branch all the way to $s_*=1$. The proof of the following fact requires substantially more insight into the problem and it will also make use of some auxiliary results, which we derive further below.

\begin{lemma} \label{lem:gron}
For all $s \in [s_0, s_*)$, we have the following a-priori upper bound
$$
\int | Q_s|^2 \lesssim 1 .
$$ 
\end{lemma}

\begin{proof}
We will derive the following the differential inequality
\begin{equation} \label{ineq:gron}
\frac{d}{ds} \int | Q_s |^2  \lesssim \int | Q_s |^2.
\end{equation}
Once this estimate is established, the desired a-priori bound follows from integrating this differential inequality. 

To show \eqref{ineq:gron}, we argue as follows. First, we note that
\begin{equation*}
\frac{d}{ds} \int | Q_s |^2 = 2 \langle Q_s, \frac{d Q_s}{ds} \rangle.
\end{equation*}
Next, by differentiating the equation satisfied by $Q_s$ with respect to $s$, we see that
\begin{equation*}
L_+ \frac{d Q_s}{ds} = - (-\DD)^s \log ( -\DD) Q_s - \frac{d \lambda_s}{ds} Q_s,
\end{equation*}
with $L_+ = (-\DD)^s + \lambda_s - (\alpha+1) |Q_s|^{\alpha}$. Recall that $\lambda_s$ is a differentiable function of $s$. Also, by bootstrap regularity arguments, we have that $Q_s \in H^{2s+1}(\R)$ and hence $(-\DD)^s \log (-\DD) Q_s \in L^2_{\mathrm{even}}(\R)$. Since $L_+$ is invertible on $L_{\mathrm{even}}^2(\R)$ and self-adjoint, we can combine the previous equations to obtain
\begin{equation}
\frac{d}{ds} \int |Q_s |^2 = I + II,
\end{equation}
where
\begin{equation} \label{eq:defII}
I = -2 \langle L_+^{-1} Q_s, (-\DD)^s \log(-\DD) Q_s \rangle , \quad II = - 2 \frac{d \lambda_s}{ds} \langle Q_s, L_+^{-1} Q_s \rangle .
\end{equation}
We start by estimating the term $I$ from above. Here a calculation shows that\footnote{Note that $R \in L^2(\R)$ by the decay estimate in Proposition \ref{prop:Q}. Moreover, we easily deduce that $R \in H^{2s+1}(\R)$ by analogous bootstrap arguments as done for $Q$.}
\begin{equation} \label{def:R}
L_+ R = -2s\lambda_s Q_s, \quad \mbox{with} \; R := \frac{d}{d \beta} \Big |_{\beta=1} \beta^{\frac{2s}{\alpha}} Q_s(\beta x) = \frac{2s}{\alpha} Q_s + x Q_s'.
\end{equation}
Therefore we conclude that
\begin{align*}
I & = \frac{1}{s \lambda_s} \langle R, (-\DD)^s \log(-\DD) Q_s \rangle = \frac{1}{s \lambda_s} \left \langle \frac{d}{d \beta} \Big |_{\beta=1} \beta^{\frac{2s}{\alpha}} Q_s (\beta \cdot), (-\DD)^s \log(-\DD) Q_s \right \rangle   \\
& = \frac{1}{2s \lambda_s}  \frac{d}{d \beta} \Big |_{\beta = 1}  \left \langle \beta^{\frac{2s}{\alpha}} Q_s(\beta \cdot), (-\DD)^s \log (-\DD) \beta^{\frac{2s}{\alpha}} Q_s(\beta \cdot) \right  \rangle \\
& = \frac{1}{2s \lambda_s} \frac{d}{d \beta} \Big |_{\beta = 1}  \left ( \beta^{\frac{4s}{\alpha} + 2s -1} \int |\widehat{Q}_s(\xi)|^2 |\xi|^{2s} \log (\beta^2 |\xi|^2 ) \, d \xi \right ) \\\
& = \frac{1}{2s\lambda_s} \left ( \left ( \frac{4s}{\alpha} + 2s -1 \right ) \big \langle Q_s, (-\DD)^s \log(-\DD) Q_s \big \rangle + 2 \big \langle Q_s, (-\DD)^s Q_s \big \rangle \right ) ,
\end{align*}
In the third step above, we used the self-adjointness of $(-\DD)^s \log(-\DD)$; whereas the fourth step follows from Plancherel's identity and change of variables. Note that all the manipulations here are well-defined, thanks to the regularity of $Q_s \in H^{2s+1}(\R)$.

Next, we apply Lemma \ref{lem:log} (derived below) which shows that the a-priori upper bound $\langle Q_s, (-\DD)^s \log (-\DD) Q_s \rangle \lesssim 1$ holds. Moreover, we notice that $\frac{4s}{\alpha} + 2s -1 \geq \frac{4s_0}{\alpha} + 2s_0 -1 \gtrsim 1$ due to the condition that $\alpha < \amax(s_0)$. In summary, we deduce that
\begin{equation*}
I \lesssim \frac{1}{\lambda_s} \lesssim \int |Q_s |^2 
\end{equation*}
for $s_0 \leq s < s_*$, where the last inequality clearly follows from Lemma \ref{lem:pohobounds}.

It remains to derive an upper bound for $II$ defined in \eqref{eq:defII} above. To this end, we recall the definition of $R$ in \eqref{def:R} which shows that
\begin{align} \label{eq:IIp}
\langle Q_s, L_+^{-1} Q_s \rangle & = -  \frac{1}{2 s \lambda_s} \langle R, Q_s \rangle = -\frac{1}{4s \lambda_s} \frac{d}{d\beta} \Big |_{\beta =1} \beta^{\frac{4s}{\alpha}} \langle Q_s(\beta \cdot), Q_s(\beta \cdot) \rangle \\ &  = \frac{1}{4s \lambda_s} \left (1- \frac{4s}{\alpha} \right ) \int |Q_s|^2. \nonumber
\end{align}
Next, if we differentiate the ``Pohozaev identities'' in Lemma \ref{lem:pohoidentities} with respect to $s$, we obtain
\begin{equation} \label{eq:llpp}
\frac{d \lambda_s}{d s} \int | Q_s |^2+ \lambda_s \frac{d}{ds} \int | Q_s |^2 = \frac{1}{2s^2} \frac{\alpha}{\alpha+2} \int | Q_s |^{\alpha+2} ,
\end{equation}
using that $\frac{d}{ds} \int |Q_s|^{\alpha+2} = 0$ holds. By combining \eqref{eq:IIp} and \eqref{eq:llpp}, we deduce that
\begin{align*}
II  & =  -2 \frac{d \lambda_s}{d s} \langle Q_s, L_+^{-1} Q_s \rangle \\
& = \big (\frac{1}{2s} - \frac{2}{\alpha} \big ) \Big \{ \frac{d}{ds} \int | Q_s |^2 -   \frac{1}{2s^2 \lambda_s} \frac{\alpha}{ \alpha+2} \int | Q_s |^{\alpha+2} \Big \} \\
& \leq \big (\frac{1}{2s} - \frac{2}{\alpha} \big ) \frac{d}{ds} \int | Q_s |^2 + \frac{C}{\lambda_s} ,
\end{align*}
for some constant $C>0$ independent of $s$. In the last step, we used again that $\int | Q_s|^{\alpha+2} = \mbox{const.}$ holds. Next, we recall that $\lambda_s^{-1} \lesssim  \int | Q_s |^2$ and $I \lesssim \int | Q_s |^2$. Hence we obtain
\begin{equation*}
\frac{d}{ds} \int | Q_s |^2 = I + II \leq \left ( \frac{1}{2s} - \frac{2}{\alpha} \right ) \frac{d}{ds} \int | Q_s |^2 +C \int | Q_s |^2,
\end{equation*}
where $C > 0$ is some constant independent of $s$. Noticing again that $1-( \frac{1}{2s} - \frac{2}{\alpha} ) \geq 1-( \frac{1}{2s_0} - \frac{2}{\alpha}) \gtrsim 1$ because of the condition $\alpha < \amax(s_0)$, we conclude that \eqref{ineq:gron} holds. The proof of Lemma \ref{lem:gron} is now complete. \end{proof}

Next, we establish an a-priori upper bound on the quantity $\langle Q_s, (-\DD)^s \log (-\DD) Q_s \rangle$, which was needed in the previous proof.

\begin{lemma} \label{lem:log}
For all $s \in [s_0, s_* )$, we have
$$
\langle Q_s, (-\DD)^s \log (-\DD) Q_s \rangle \lesssim 1.
$$
\end{lemma}

\begin{proof}
From the identity $Q_s = ((-\DD)^s + \lambda_s)^{-1} |Q_s|^{\alpha} Q_s$ we deduce that
\begin{equation} \label{ineq:young1}
\| (-\DD)^t Q_s \|_2 = \left \| \frac{(-\DD)^t}{(-\DD)^s + \lambda_s} |Q_s|^{\alpha} Q_sÊ\right \|_2 \leq \| (-\DD)^{t-s} ( |Q_s|^{\alpha} Q_s) \|_2 ,
\end{equation}
for any $t \geq 0$. In particular, we can choose 
\begin{equation*}
t := s - \frac{\alpha}{4 (\alpha+2)} ,
\end{equation*}
which implies that $s > t > s-s_0/2 \geq s_0/2$ thanks to the condition $\alpha < \amax(s_0)$.
 
By our choice of $t$, the operator $(-\DD)^{t-s}$ on $\R$ is given by convolution with the singular integral kernel $|x|^{-(\alpha+4)/(2 (\alpha+2))}$ up to a multiplicative constant $C$ depending only on $\alpha$.  Hence, by the weak Young inequality, we deduce from \eqref{ineq:young1} the following bound
\begin{equation*} \label{ineq:tbound}
\| (-\DD)^{t} Q_s \|_2 \lesssim \| |x|^{-\frac{\alpha+4}{2(\alpha+2)} } \ast ( |Q_s|^{\alpha} Q_s ) \|_2 \lesssim  \| |Q_s|^{\alpha+1} \|_{\frac{\alpha+2}{\alpha+1}} \lesssim \| Q_s \|_{\alpha+2}^{\alpha+1} \lesssim 1,
\end{equation*}
using that $\int | Q_s |^{\alpha+2} = \mbox{const}.$ holds. But the last estimate implies that 
\begin{align*} 
\langle Q_s, (-\DD)^s \log (-\DD) Q_s \rangle &  =  \int |\xi|^{2s} \log(|\xi|^2) |\widehat{Q}_s(\xi)|^2 \, d\xi \\ & \lesssim \int |\xi|^{4t} |\widehat{Q}_s(\xi)|^2 \, d\xi \lesssim \| (-\DD)^t Q_s \|_2^2  \lesssim 1.
\end{align*}
Here we used Plancherel's identity together with the inequality
\begin{equation} \label{ineq:log2}
 \log(|\xi|^{2}) \leq C |\xi|^{4t-2s} ,
\end{equation}
where the constant $C > 0$ only depends on $\alpha$ and $s_0$. Indeed, this inequality can be simply derived as follows. Note that
$$
4t-2s = 2s - \frac{\alpha}{\alpha+2} \geq  2s_0 - \frac{\alpha}{\alpha+2} \geq \delta,
$$
for some constant $\delta > 0$ depending only on $\alpha$ and $s_0$. (To see this, simply use the strict inequality $\frac{\alpha}{\alpha+2} < 2s_0$ due to the condition on $\alpha$.)  Since $\log (z^2) \leq 2\delta^{-1} z^{\delta}$ for $z \geq 1$ and $\delta \leq 4t-2s$, we deduce that \eqref{ineq:log2} holds for $|\xi| \geq 1$. The inequality \eqref{ineq:log2} is obviously true when $|\xi| < 1$, since the left-hand side is negative in this case. This completes the proof of Lemma \ref{lem:log}. \end{proof}

As a next step, we wish to analyze the sequences $\{ Q_{s_n} \}_{n=0}^\infty$ where $s_n \to s_*$. In particular, our goal is to derive {\em strong convergence} of $\{ Q_s \}_{n=1}^\infty$ (along the subsequences) with respect to the norm $\| \cdot \|_{X^\alpha} = \| \cdot \|_2 + \| \cdot \|_{\alpha+2}$. Recall from Lemmas \ref{lem:pohobounds} and \ref{lem:gron} the a-priori bound 
$$
\int |(-\DD)^{\frac{s}{2}} Q_s|^2 + \int |Q_s|^2 \lesssim 1, \quad \mbox{for $s \in [s_0,s_*)$}.
$$ 
Suppose now that $s_n \to s_*$. To turn the uniform bound $\| Q_{s_n} \|_{H^s} \lesssim 1$ into strong convergence of $\{ Q_{s_n} \}_{n=1}^\infty$  in some $L^p$-norm, we need a further ingredient. Indeed, since we consider $d=1$ space dimension, we recall the well-known fact that the even-symmetry of the functions $\{ Q_{s_n} \}_{n=1}^\infty$ (unlike for radial symmetry in $d \geq 2$ dimensions) is {\em not sufficient} to gain relative compactness of $\{Q_{s_n}\}_{n=1}^\infty$ in some $L^p$-norm. To deal with this, we will now impose that $Q_{s_0}=Q_{s_0}(|x|)> 0$ is  a positive function. Then the following result shows that $Q_s( |x| ) > 0$ along the branch. This fact will, in turn, lead to monotonicity of the functions $Q_s(|x|)$ in $|x|$. From this property and the a-priori bound on $\lambda_s \sim 1$, we finally derive a {\em uniform decay} estimate of the form $Q_s(|x|) \lesssim |x|^{-1}$ for $|x|$ outside a fixed compact set, which will enable us to gain relative compactness in $L^2(\R) \cap L^{\alpha+2}(\R)$. 

First, we establish that positivity of $Q_s(|x|) > 0$ holds along the maximal branch, provided that $Q_{s_0}(|x|) > 0$ is assumed initially.

\begin{lemma} \label{lem:Qposi}
Suppose that $Q_{s_0}=Q_{s_0}(|x|) > 0$ is positive. Then $Q_{s} = Q_{s}(|x|) > 0$ for $x \in \R$ and $s \in [s_0,s_*)$. 
\end{lemma}

\begin{proof}
We divide the proof into two steps as follows.

\subsubsection*{Step 1} First, we show that positivity of $Q_{s}(|x|) > 0$ is an ``open'' property along the branch $(Q_s, \lambda_s)$. That is, if we assume that $Q_{\tilde{s}} = Q_{\tilde{s}}(|x|) > 0$ for some $\tilde{s} \in [s_0, s_*)$ then 
$$
\mbox{$Q_{s}= Q_s(|x|) > 0$ for $s \in [s_0,s_*)$ with $|s - \tilde{s} | < \eps$,}
$$
where $\eps > 0$ is sufficiently small. To prove this claim, we consider the family of self-adjoint operators
$$
A_{s} = (-\DD)^{s} + \lambda_s - V_s, \quad \mbox{with $V_s = |Q_s|^\alpha$,}
$$
acting on $L^2(\R)$. Clearly, we have 
\begin{equation} \label{eq:AsQs}
A_s Q_s = 0
\end{equation}
In particular, we see that $Q_s$ is an eigenfunction of $A_s$ with eigenvalue 0. Furthermore, by Lemma \ref{lem:perron}, the lowest eigenvalue of $A_s$ is nondegenerate and its corresponding eigenfunction is strictly positive. In particular, the function $Q_{\tilde{s}}(|x|) > 0$ is the ground state eigenfunction of $A_{\tilde{s}}$ and hence 0 is the lowest eigenvalue of $A_{\tilde{s}}$. Thus, in view of \eqref{eq:AsQs} and Lemma \ref{lem:perron}, it suffices to show that $0$ is the lowest eigenvalue of $A_s$  (for $s$ close to $\tilde{s}$) and finally rule out that $Q_{s} < 0$ holds. 

To deduce that $0$ is the lowest eigenvalue of $A_s$ when $s$ is close to $\tilde{s}'$, we use a spectral convergence argument. Indeed, we claim that $A_s \to A_{\tilde{s}}$ in the norm resolvent sense as $s\to \tilde{s'}$, by which we mean that
\begin{equation} \label{eq:Asconv}
 \left \| \frac{1}{A_s + z} - \frac{1}{A_{\tilde{s}} +z} \right \|_{L^2 \to L^2} \to 0
\end{equation} 
as $s \to \tilde{s}$, where $z \in \mathbb{C}$ with $\mathrm{Im} \, z \neq 0$. In fact, by straightforward estimates, we find that \eqref{eq:Asconv} will follow from $\lambda_s \to \lambda_{\tilde{s}}$ and provided we can show that $V_s = |Q_s|^\alpha$ satisfies
\begin{equation} \label{conv:Vs}
\mbox{$\| V_s - V_{\tilde{s}} \|_p \to 0$ as $s \to \tilde{s}$},
\end{equation} 
for some $p \geq 1$ such that $p > \frac{1}{2s_0} \geq \frac{1}{2s}$. 

To see that we can always find such $p$, we argue as follows.  Recall that $V_s = |Q_s|^\alpha$ and $Q_s \to Q_{\tilde{s}}$ in $L^2(\R)$. Moreover, by Lemmas \ref{lem:pohobounds} and \ref{lem:gron}, we have $\| Q_s \|_{H^s} \lesssim 1$ and thus $\| Q_{s} \|_{H^{s_0}} \lesssim 1$ since $s \geq s_0$. Hence, by interpolation, we find that 
$$
\mbox{$\| Q_s - Q_{\tilde{s}} \|_{H^{\sigma_0}} \to 0$ as $s \to \tilde{s}$},
$$ 
for any $0 \leq \sigma_0 < s_0$. In particular if $s_0 > 1/2$, then $\| Q_s - Q_{\tilde{s}} \|_\infty \to 0$ by Sobolev inequalities. Hence we can choose $p=+\infty$ in \eqref{conv:Vs} and we conclude that \eqref{eq:Asconv} holds whenever $s_0 > 1/2$. Assume now that $s_0 \leq 1/2$. In this case, by Sobolev inequalities and H\"older's inequality, we deduce that $V_s = |Q_s|^\alpha$ satisfies \eqref{conv:Vs} for $p = \frac{1}{\alpha} \frac{2}{1-2\sigma_0}$ with any $0 \leq \sigma_0 < s_0$. But since $\frac{1}{\alpha} \frac{2}{1-2 s_0} > \frac{1}{2s_0}$ due to the condition $\alpha < \amax(s_0)$, we can choose $\sigma_0 < s_0$ sufficiently close to $s_0$ such that $p = \frac{1}{\alpha} \frac{2}{1-2 \sigma_0} > \frac{1}{2 s_0}$ as well. This shows that \eqref{eq:Asconv} also holds when $s_0 \leq 1/2$. 
    
Since we have derived that $A_s \to A_{\tilde{s}}$ in the norm resolvent sense, we can now complete the proof by standard spectral arguments: Let $\lambda_1(A_s)$ denote the lowest eigenvalue of $A_s$. By Lemma \ref{lem:perron}, the eigenvalue $\lambda_1(A_s)$ is nondegenerate and its corresponding eigenfunction $\psi_{1,s}(x) > 0$ is strictly positive. Since $Q_{\tilde{s}}(x) > 0$ satisfies $A_{\tilde{s}} Q_{\tilde{s}} = 0$, we deduce that $\lambda_1(A_{\tilde{s}}) = 0$ and $Q_{\tilde{s}}(x) = \psi_{1,\tilde{s}}(x)$ holds. Since $A_s \to A_{\tilde{s}}$ in the norm resolvent sense, we conclude that $\lambda_1(A_s) \to \lambda_1(A_{\tilde{s}})$ as $s \to \tilde{s}$ and that $\lambda_1(A_s)$ is simple for $s$ close to $\tilde{s}$. (The last statement also follows from Lemma \ref{lem:perron}.) Moreover note that $\lambda_1(A_{\tilde{s}})$ is isolated. Hence we can find $c > 0$ sufficiently small such that the interval $I_c = (-c, c)$ satisfies $\sigma  (A_{\tilde{s}}) \cap I_c = \{ \lambda_1(A_{\tilde{s}}) \}$. Thus, by the above convergence properties, we deduce that 
$$ \sigma(A_s) \cap I_c = \{ \lambda_1(A_s) \} $$
whenever $|s - \tilde{s} | <\eps$, where $\eps > 0$ is sufficiently small. On the other hand, we recall that $A_{s} Q_{s} = 0$, which shows that $\lambda_1(A_s) = 0$ holds for $|s-\tilde{s}| <\eps$. By Lemma \ref{lem:perron} again, we deduce that $Q_{s}(x) = \sigma(s) \psi_{1,s}(x)$, where $\psi_{1,s}(x) > 0$ is the unique ground state eigenfunction of $A_{s}$ and $\sigma(s) \in \{+1, -1 \}$ is some sign depending on $s$. However, we have that $Q_s(x) \to Q_{\tilde{s}}(x) > 0$ pointwise a.\,e.~as $s \to \tilde{s}$, which implies that $\sigma(s) = +1$ for all $s$ close to $s_0$. Therefore we conclude that $Q_{s}(x) = \psi_{1,s}(x) > 0$ for all $|s-\tilde{s}| < \eps$, provided that $\eps > 0$ is small and $Q_{\tilde{s}}(x) > 0$ holds. 

\subsubsection*{Step 2} Next, we prove that positivity of $Q_s$ along the branch is a ``closed'' property. That is, if $Q_{s}(|x|) > 0$ for all $s \in [s_0, \tilde{s})$ with some $\tilde{s} < s_*$, then $Q_{\tilde{s}}(|x|) > 0$ as well. Indeed, let $\tilde{s} \in (s_0, s_*)$ be given and suppose that $\{ s_n \}_{n=1}^\infty \subset [s_0,\tilde{s})$ is a sequence with $s_n \to \tilde{s}$. Moreover, we assume that $Q_{s_n}(|x|) > 0$ for all $n \in \mathbb{N}$. Note that $Q_{s_n} \to Q_{\tilde{s}}$ strongly in $H^{\sigma_0}(\R)$ for any $0 \leq \sigma_0 < s_0$, as shown in Step 1 above. In particular, we have that $Q_{s_n} \to Q_{\tilde{s}}$ pointwise a.\,e.~in $\R$, which implies  $Q_{\tilde{s}}(|x|) \geq 0$. Also, notice that $Q_{\tilde{s}} \not \equiv 0$ due to $\| Q_{\tilde{s}} \|_{\alpha+2} = \| Q_{s_0} \|_{\alpha+2} \neq 0$. Hence $Q_{\tilde{s}}$ is a nonnegative and nontrivial solution of
$$
Q_{\tilde{s}} = \frac{1}{(-\DD)^{\tilde{s}} + \lambda_{\tilde{s}}} |Q_s|^\alpha Q_s.
$$
From this we deduce that positivity $Q_{\tilde{s}}(|x|) > 0$ holds by using  Lemma \ref{lem:resolvent_p}, which establishes the positivity of the integral kernel for the resolvent $((-\DD)^s + \lambda)^{-1}$ with $0 < s < 1$ and $\lambda > 0$.

By combining the results of Steps 1 and 2 above, we complete the proof of Lemma~\ref{lem:Qposi}. \end{proof}

Next, we derive a uniform spatial decay estimate along the maximal branch $(Q_s, \lambda_s)$, provided that $Q_{s_0}(|x|) > 0$ holds initially.

\begin{lemma} \label{lem:decay}
Suppose that $Q_{s_0}(|x|) > 0$ holds. Then we have the uniform decay estimate
$$
0 < Q_{s}(|x|) \lesssim \frac{1}{|x|} 
$$
for $|x| \geq R_0$ and $s \in [s_0, s_*)$. Here $R_0 > 0$ is some constant independent of $s$.
\end{lemma}

\begin{proof}
For any $\mu > 0$ given, we can rewrite the equation satisfied by $Q_s$ as follows: 
$$
Q_s = \big ( (-\DD)^s + \mu \big)^{-1} f_s , \quad \mbox{with $f_s(x)  = Q_s(x) \big ( Q_s^{\alpha}(x) - \lambda_s + \mu \big ).$}
$$
Note that $|Q|^\alpha_s = Q^\alpha_s$, since $Q_s(|x|) > 0$ for $s \in [s_0,s_*)$ by Lemma \ref{lem:Qposi}.

By Lemmas \ref{lem:pohobounds} and \ref{lem:gron}, we have the uniform lower bound $\lambda_s \gtrsim 1$. In particular, we can choose $\mu > 0$ fixed and independent of $s$ such that $\lambda_s \geq 2\mu$ for $s\in [s_0, s_*)$. Next, we claim that the positive part $f_{s}^+ := \max \{ f_s, 0\}$ has compact support such that 
\begin{equation} \label{eq:fsupp}
f_{s}^+(x) \equiv 0 \quad \mbox{for $|x| \geq \frac{1}{2} R_0$},
\end{equation}
 where $R_0>0$ is some large constant independent of $s$. Indeed, the functions $Q_s = Q_s(|x|) > 0$ are even and positive. Hence, by Lemma \ref{lem:symm}, we deduce that each function $Q_s(|x|)$ is strictly decreasing in $|x|$. Also, we recall the uniform bound $\| Q_s \|_2 \lesssim 1$ from Lemma \ref{lem:gron}. Hence, for any $|x| > 0$, 
  $$
 |x| |Q_s(x)|^2 \leq \frac{1}{2} \int_{|y| \leq |x|} |Q_s(y)|^2 \, dy \leq \frac{1}{2} \int |Q_s(y)|^2 \,dy \lesssim 1.
 $$ 
Therefore $Q_s(y) \lesssim |x|^{-1/2}$ for $|x| > 0$. Moreover, we have $-\lambda_s + \mu \leq -\mu < 0$ and $Q_s(|x|) > 0$. These facts imply that \eqref{eq:fsupp} holds with some large constant $R_0 >0$ independent of $s$.

Next, by Lemma \ref{lem:resolvent_p}, we conclude that the kernel $G_{s,\mu}$ of the resolvent $((-\DD)^s+\mu)^{-1}$ is given by a positive function $G_{s,\mu}(x)>0$ that satisfies the uniform bound
\begin{equation}  \label{ineq:green}
 0 < G_{s,\mu}(x) \lesssim \frac{1}{|x|} \quad \mbox{for $|x| >0$}.
\end{equation}
Since $|x-y| \geq \frac{1}{2} |x|$ when $|x| \geq R_0$ and $|y| \leq \frac{1}{2} R_0$, we can combine \eqref{eq:fsupp} and \eqref{ineq:green} to find the following bound 
\begin{equation*}
0 < Q_s(|x|)  \leq \int_{|y| \leq \frac{1}{2} R_0} G_{s,\mu}(x-y) f_{s}^+(y) \, dy \lesssim \frac{1}{|x|} \int_{|y| \leq \frac{1}{2}R_0} f^+_{s}(y) \, dy  \lesssim  \frac{1}{|x|} ,
\end{equation*}
for $|x| \geq R_0$. In the last step, we used the uniform bounds $\| Q_s \|_{\alpha+2} \lesssim 1$ and $\| Q_s \|_2 \lesssim 1$ together with H\"older's inequality to obtain that
$$
\int_{|y| \leq \frac{1}{2} R_0} f_s^+(y) \, dy \leq  R_0^{\frac{1}{\alpha+2}} \| Q_s \|^{\alpha+1}_{\alpha+2} + R_0^{\frac{1}{2}} |\lambda_s - \mu|  \| Q_s \|_2 \lesssim 1.
$$
This completes the proof of Lemma \ref{lem:decay}. \end{proof}

We are now in the position to derive the following key fact.

\begin{lemma} \label{lem:cv}
Let $\{ s_n \}_{n=1}^\infty \subset [s_0,s_*)$ be a sequence such that $s_n \to s_*$. Furthermore, we suppose that $Q_{s_n} = Q_{s_n}(|x|) > 0$ are positive functions. Then (after possibly passing to a subsequence) we have
$$
\mbox{$Q_{s_n} \to Q_*$ in $L^2(\R) \cap L^{\alpha+2}(\R)$ and $\lambda_{s_n} \to \lambda_*$}, 
$$
where $\lambda_* > 0$ and $Q_*=Q_*(|x|) > 0$ satisfy
$$
(-\DD)^{s_*} Q_* + \lambda_* Q_* - Q_*^{\alpha+1} = 0 .
$$
\end{lemma}

\begin{remarks}
{\em 1.) 
One key step in the proof of Lemma \ref{lem:cv} will be to establish {\em strong convergence} of $\{ Q_{s_n} \}_{n=1}^\infty$ in $L^2(\R)$. Here, the pointwise decay bound from Lemma \ref{lem:decay} will guarantee this fact. Note that the (weaker) uniform decay estimate $Q_{s_n}(x) \lesssim |x|^{-1/2}$ (see proof of Lemma \ref{lem:decay}) is not sufficient to conclude strong convergence of $\{ Q_{s_n} \}_{n=1}^\infty$ in $L^2(\R)$. Abstractly speaking, the gain of relative compactness of $\{ Q_n \}_{n=1}^\infty$ in $L^2(\R)$ is due to the fact that the $Q_{s_n}$ solve equation \eqref{eq:phi} with a uniform bound on the nonlinear eigenvalues $\lambda_{s_n} \sim 1$. 

2.) By bootstrapping arguments, we can in fact derive strong convergence of $\{ Q_{s_n} \}_{n=1}^\infty$ in $H^{2s_*}(\R)$, once strong convergence in $L^2(\R) \cap L^{\alpha+2}(\R)$ is known. However, we do not need this refinement in the following. Hence we omit its proof.
}
\end{remarks}

\begin{proof}
Define the sequences $\{ Q_n \}_{n=1}^\infty$ with $Q_n = Q_{s_n}$ and $\{ \lambda_n \}_{n=1}^\infty$ with $\lambda_n = \lambda_{s_n}$.

First, by combining Lemmas \ref{lem:pohobounds} and \ref{lem:gron}, we obtain the uniform bound $1 \lesssim \lambda_n \lesssim 1$. Thus, after passing to a subsequence, we can assume that $\lambda_n \to \lambda_*$ with some positive limit $\lambda_* > 0$.

From Lemmas \ref{lem:pohobounds} and \ref{lem:gron} we have  the a-priori bound $\| Q_n \|_{H^{s_n}} \lesssim 1$. Since $s_n \geq s_0$, this implies in particular that $\| Q_n \|_{H^{s_0}} \lesssim 1$ holds. Hence (after passing to a subsequence) we can assume that $Q_n \weakto Q_*$ weakly in $H^{s_0}(\R)$ and $Q_n(x) \to Q_*(x)$ pointwise a.\,e.~in $\R$. Moreover, by local Rellich-Kondratchov compactness, we deduce that $Q_n \to Q_*$ in $L^2_{\mathrm{loc}}(\R)$. To upgrade this fact to strong convergence in $L^2(\R)$ itself, we recall that Lemma \ref{lem:decay} implies the uniform decay estimate
\begin{equation} \label{ineq:uniformL2}
|Q_n(x)| \lesssim \frac{1}{|x|} Ê\quad \mbox{for $|x| \geq R_0$ and $n \geq 1$}.
\end{equation}
where $R_0 > 0$ is independent of $n$. Using this uniform decay, we easily derive strong convergence of $\{ Q_n \}_{n=1}^\infty$ in $L^2(\R)$. Indeed, let $\eps > 0$ be given. Choose $R_\eps \geq R_0$ large enough such that $\int_{|x| > R_\eps} |x|^{-2} \leq \eps^2$ and $\int_{|x| > R_\eps} |Q_*|^2 \leq \eps^2$. Since moreover $Q_n \to Q_*$ in $L^2_{\mathrm{loc}}(\R)$, there exists $n_0  \geq 1$ such that $\int_{|x| \leq R_\eps} |Q_n- Q_*|^2 \leq \eps^2$ for $n \geq n_0$. Using the pointwise bound \eqref{ineq:uniformL2} and the triangle inequality, we thus conclude
$$
\| Q_n - Q_* \|_{L^2(\R)} \leq \| Q_n - Q_* \|_{L^2(|x| \leq R_{\eps})}+ \| Q_n - Q_* \|_{L^2(|x| > R_\eps)} \lesssim \eps, $$
for all $n \geq n_0$. This shows that $Q_n \to Q_*$ strongly in $L^2(\R)$. 

To see that $Q_n \to Q_*$ strongly in $L^{\alpha+2}(\R)$, we first recall the uniform bound $\| Q_n \|_{H^{s_0}} \lesssim 1$. Using the condition $\alpha < \amax(s_0)$ and Sobolev inequalities, we deduce that $\| Q_n \|_p \lesssim 1$ for some $p > \alpha+2$ (with $p < \frac{2}{1-2s_0}$ if $s_0 \leq 1/2$). Thus, by interpolation, we deduce that $Q_n \to Q_*$ in $L^{\alpha+2}(\R)$ as well.

Finally, we show that the limit $Q_*=Q_*(|x|) > 0$ is a positive function which satisfies
\begin{equation} \label{eq:phistar}
(-\DD)^{s_*} Q_* + \lambda_* Q_* - Q_*^{\alpha+1}  = 0.
\end{equation} 
Indeed, the latter fact simply follows from passing to the limit in the equation satisfied by $Q_n$ together with the convergence properties derived above. Since $Q_n = Q_n(|x|) > 0$ and $Q_n(x) \to Q_*(x)$ pointwise a.\,e. in $\R$, we find that $Q_* = Q_*(|x|) \geq 0$ holds. Note that $Q_n \to Q_*$ in $L^{\alpha+2}(\R)$ and $\| Q_n \|_{\alpha+2} = \| Q_0 \|_{\alpha+2} \neq 0$ for all $n \in \mathbb{N}$. Hence $Q_* \not \equiv 0$ as well. Finally, we deduce positivity $Q_*(x) > 0$ by noting that $Q_* = ((-\DD)^{s_*} + \lambda_*)^{-1} Q_*^{\alpha+1}$ and using the positivity of the integral kernel of the resolvent $((-\DD)^{s_*} + \lambda_*)^{-1}$; see Lemma \ref{lem:resolvent_p}. 

This proof of Lemma \ref{lem:cv} is now complete. \end{proof}

As the one of the main results of this section, we now prove that any maximal branch $(Q_s, \lambda_s)$ extends to $s_*=1$, provided that $Q_{s_0}$ satisfies some explicit conditions (which in particular hold true if $Q_{s_0}$ is a ground state). 

\begin{prop}  \label{prop:global}
Let $0 < s_0 < 1$ and $0 < \alpha < \amax(s_0)$ be given. Suppose that $(Q_0, \lambda_0) \in X^\alpha \times \Rplus$ satisfies Assumption \ref{ass:branch} with $s=s_0$ and $\lambda = \lambda_0$. Furthermore, assume that $Q_{0} = Q_0(|x|) > 0$ is positive and that the corresponding linearized operator $L_{+,0}$ satisfies $\cN_{-,\mathrm{even}}(L_{+,0}) = 1$. 

Then the corresponding maximal branch $(Q_s, \lambda_s) \in C^1([s_0, s_*); X^\alpha \times \Rplus)$ extends to $s_*=1$. Moreover, we have that
$$
\mbox{$Q_s \to Q_*$ in $L^2(\R) \cap L^{\alpha+2}(\R)$ and $\lambda_s \to \lambda_*$ as $s \to 1$},
$$
where $Q_* = Q_*(|x|) > 0$ is the unique solution of
\begin{equation*}
\left \{ \begin{array}{l} -\Delta Q_* + \lambda_* Q_* - Q_*^{\alpha+1} = 0, \\
Q_* = Q_*(|x|) > 0, \quad Q_* \in L^2(\R) \cap L^{\alpha+2}(\R)  , \end{array} \right .
\end{equation*}
and $\lambda_* > 0$ is given by
\begin{equation*} \label{eq:lambda_limit}
\lambda_* = \left (  \frac{\alpha}{2(\alpha+2)}   \frac{  \int |Q_{0}|^{\alpha+2} } {  \int |\nabla P|^2 } \right )^{\frac{2\alpha}{\alpha+4}} .
\end{equation*}
Here $P=P(|x|)> 0$ denotes the unique positive, even solution in $C^2(\R)$ that satisfies $-\Delta P+ P - P^{\alpha+1} = 0$ with $P \to 0$ as $|x| \to \infty$.
\end{prop}
 
\begin{remarks}{\em 
1.) The unique solution $$P(x)=\frac{(\sigma+1)^\frac{1}{2\sigma}}{\cosh^{\frac{1}{\sigma}}(\sigma x)}$$ where $\sigma = \frac{\alpha}{2}$ is known in closed form. However, this fact has no relevance in the proof below.

2.) Note that $\lambda_* > 0$ only depends on $\alpha$ and the quantity $\int | Q_{0} |^{\alpha+2}$. 
}
\end{remarks}

\begin{proof}
Let $(Q_s, \lambda_s) \in C^1([s_0,s_*); X^\alpha \times \Rplus)$ be the maximal branch with $s_* \in (s_0, 1]$. By Lemma \ref{lem:Qposi}, the functions $Q_s = Q_s(|x|) > 0$ are positive for all $s \in [s_0,s_*)$. Next, we consider the linearized operators along the branch, i.\,e.,
$$
L_{+,s} = (-\DD)^s + \lambda_s - (\alpha+1) Q_s^{\alpha} .
$$ 
We claim that the Morse index of $L_{+,s}$ acting on $L^2_{\mathrm{even}}(\R)$ is constant. That is,
$$
\cN_{-,\mathrm{even}} (L_{+,s}) = 1, \quad \mbox{for $s \in [s_0,s_*)$}.
$$ 
Indeed, this follows from the initial assumption that $\cN_{-,\mathrm{even}} (L_{+,s_0}) = 1$ and a continuity argument: Similar as in the proof of Lemma \ref{lem:Qposi}, we deduce that $L_{+,s} \to L_{+,\tilde{s}}$ in the norm resolvent sense as $s \to \tilde{s}$. Hence, by continuity of eigenvalues of $L_{+,s}$, any change of the Morse index along the branch would imply that 0 must be an eigenvalue of $L_{+,s}$ acting on $L^2_{\mathrm{even}}(\R)$ for some $s \in (s_0, s_*)$. But this contradicts Assumption \ref{ass:branch}.

Suppose now that $\{ s_n \}_{n=1}^\infty \subset [s_0, s_*)$ be a sequence such that $s_n \to s_*$. Define the sequences $\{ Q_{n} \}_{n=1}^\infty \subset X^\alpha$ and $\{ \lambda_n \}_{n=1}^\infty \subset \Rplus$ by $Q_n = Q_{s_n}$ and $\lambda_n = \lambda_{s_n}$ for $n \in \N$. 

Next, by Lemma \ref{lem:cv} and after passing to a subsequence if necessary, we can assume that $Q_n \to Q_*$ in $L^2(\R)\cap L^{\alpha+2}(\R)$ and $\lambda_n \to \lambda_*$ for some $Q_*(|x|) >0$ and $\lambda_* > 0$ satisfying
$$
(-\DD)^{s_*} Q_* + \lambda_* Q_* - Q_*^{\alpha+1} = 0.
$$

Next, we prove that $s_* = 1$ holds. Suppose on the contrary that $s_* < 1$ was true. We consider the sequence $\{ L_{+,n} \}_{n=1}^\infty$ of self-adjoint operators given by
$$
L_{+,n} = (-\DD)^{s_n} + \lambda_n - (\alpha+1) Q_n^{\alpha} ,
$$
acting on $L^2_{\mathrm{even}}(\R)$. Note that $\|Q_n \|_{H^{s_0}} \lesssim 1$ and $Q_n \to Q_*$ in $L^2(\R)$. Then, by adapting the proof of Lemma \ref{lem:Qposi}, we deduce that $Q_{n} \to Q_*$ in $L^p(\R)$ with some $p > 1/2s_0 \geq 1/2s_*$. In particular, this implies that $L_{+,n} \to L_{+,*}$ in norm resolvent sense, where
$$
L_{+,*} = (-\DD)^{s_*} + \lambda_* - (\alpha+1) Q_*^{\alpha}.
$$
Since the Morse index is lower semi-continuous with respect to the norm resolvent topology, we conclude that 
$$1 = \liminf_{n \to \infty} \cN_{-,\mathrm{even}} (L_{+,n}) \geq \cN_{-, \mathrm{even}} (L_{+,*}).$$ On the other hand, we easily calculate that $\langle Q_*, L_{+,*} Q_* \rangle = -\alpha \int |Q_*|^{\alpha+2} < 0$. By the min-max principle, we deduce that $L_{+,*}$ acting on $L^2_{\mathrm{even}}(\R)$ has at least one eigenvalue that is strictly negative. Thus we conclude that
$$
\cN_{-,\mathrm{even}}  (L_{+,*}) =  1.
$$
Now we can apply Lemma \ref{lem:kernel} to deduce that $L_{+,*}$ is invertible on $L^2_{\mathrm{even}}(\R)$. Hence $(Q_*,\lambda_*) \in X^\alpha \times \Rplus$ satisfies Assumption \ref{ass:branch}. Hence, by Proposition \ref{prop:local}, we can extend the branch $(Q_s, \lambda_s)$ beyond $s_*$, which contradicts the maximality property of $s_*$.  Hence the assumption $s_* < 1$ leads to a contradiction.

Now we have shown that $s_* = 1$ holds. By Lemma \ref{lem:cv}, we see that $Q_n \to Q_*$ in $L^2(\R) \cap L^{\alpha+2}(\R)$ and $\lambda_n \to \lambda_* > 0$, where $Q_* = Q_*(|x|) > 0$ solves the nonlinear equation
\begin{equation} \label{eq:ode}
- \Delta Q_* + \lambda_* Q_*  - Q_*^{\alpha+1} = 0 .
\end{equation}
Note that, by bootstrapping this equation for $Q_* \in L^2(\R) \cap L^{\alpha+2}(\R)$, we conclude that $Q_* \in H^2(\R)$. Using this fact , we deduce that in fact $Q_* \in C^2(\R)$ holds. Next, we recall the well-known fact that $-\Delta P + P -  P^{\alpha+1} = 0$ has a unique positive solution $P = P(|x|) > 0$ in $C^2(\R)$ with $P \to 0$ as $|x| \to \infty$; see the remark above. By a simple scaling argument, we infer that
\begin{equation} \label{eq:unique}
Q_*(x) = \lambda^{\frac{1}{\alpha}}_* P(\lambda_*^{\frac{1}{2}} x) ,
\end{equation}
Next, we integrate \eqref{eq:ode} against $\frac{1}{2} Q_* + x \cdot \nabla Q_*$, which gives the Pohozaev identity
\begin{equation} \label{eq:pohos}
  \int | \nabla Q_* |^2 = \frac{\alpha}{2 (\alpha+2)} \int | Q_* |^{\alpha+2}.
\end{equation}
Since $\int | Q_n |^{\alpha+2} = \int | Q_{s_0} |^{\alpha+2}$ for all $n \geq 1$ and $Q_n \to Q_*$ in $L^{\alpha+2}(\R)$, we find that 
\begin{equation} \label{eq:alplimit}
\int | Q_* |^{\alpha+2} = \int | Q_{s_0} |^{\alpha+2}.
\end{equation}
Using now \eqref{eq:unique}, \eqref{eq:pohos} and \eqref{eq:alplimit}, an elementary calculation shows that $\lambda_* > 0$ is given by the formula in Proposition \ref{prop:global}. In particular, this shows that the limit $\lambda_* >0$ is independent of the sequence $\{ s_n \}_{n=1}^\infty$. Furthermore, by uniqueness of $Q_*(x)$ with $\lambda_* > 0$ given, we conclude that the limit $Q_* \in H^2(\R)$ is also independent of $\{s_n\}_{n=1}^\infty$. Hence $Q_s \to Q_*$ in $L^2(\R) \cap L^{\alpha+2}(\R)$ and $\lambda_s \to \lambda_*$ as $s \to 1$. 

The proof of Proposition \ref{prop:global} is now complete.\end{proof}

\subsection{Proof of Theorem \ref{thm:unique}} 

\label{subsec:unique}

First, we prove uniqueness of ground states and we argue by contradiction as follows.

Let $0 < s_0 < 1$ and $0 < \alpha < \amax(s_0)$ be given. Recall our definition of the real Banach space $X^\alpha$ of real-valued and even functions in $L^2(\R) \cap L^{\alpha+2}(\R)$; see \eqref{def:X}. 

Suppose that $Q = Q_0(|x|) > 0$ and $\widetilde{Q}_0 =\widetilde{Q}_0(|x|) > 0$ are two ground states for problem \eqref{eq:phi} such that $Q_0 \not \equiv \widetilde{Q}_0$. By Theorem \ref{thm:nondeg}, we have that $Q_0 \in X^\alpha$ and $\widetilde{Q}_0 \in X^\alpha$ both satisfy Assumption \ref{ass:branch} with $s=s_0$ and $\lambda = 1$. Hence, by Proposition \ref{prop:global}, there exist two global branches 
$$\mbox{$(Q_s, \lambda_s) \in C^1([s_0,1); X^\alpha \times \Rplus)$ and $(\widetilde{Q}_s, \widetilde{\lambda}_s) \in C^1([s_0,1); X^\alpha \times \Rplus)$},$$
 which solve equation \eqref{eq:phis} and we have $(Q_{s_0}, \lambda_{s_0)} = (Q_0, 1)$ and $(\widetilde{Q}_{s_0}, \widetilde{\lambda}_{s_0}) = (\widetilde{Q}_0, 1)$. Note that, by  the local uniqueness stated in Proposition \ref{prop:local}, the branches $(Q_s, \lambda_s)$ and $(\widetilde{Q}_s, \widetilde{\lambda}_s)$ cannot intersect. Moreover, by Proposition \ref{prop:global}, we have the following facts.
\begin{itemize}
\item $Q_s \to Q_*$ in $L^2(\R) \cap L^{\alpha+2}(\R)$ and $\lambda_s \to \lambda_*$ as $s \to 1$.
\item $\widetilde{Q}_s \to \widetilde{Q}_*$ in $L^2(\R)\cap L^{\alpha+2}(\R)$ and $\widetilde{\lambda}_s \to \widetilde{\lambda}_*$ as $s \to 1$.
\end{itemize}
Here $\lambda_* > 0$ and $\widetilde{\lambda}_* > 0$ are given by the formula in Proposition \ref{prop:global}. Furthermore, the functions $Q_* = Q_*(|x|) > 0$ and $\widetilde{Q}_*=\widetilde{Q}_*(|x|) > 0$ are the unique even and positive solutions in $L^2(\R) \cap L^{\alpha+2}(\R)$ of the nonlinear equations
$$
-\Delta Q_* + \lambda_* Q_* - Q_*^{\alpha+1} = 0, \quad -\Delta \widetilde{Q}_* + \widetilde{\lambda}_* \widetilde{Q}_* - \widetilde{Q}_*^{\alpha+1} =  0,
$$ 
respectively. In view of Proposition \ref{prop:global}, we can conclude the equality
$$\lambda_* = \widetilde{\lambda}_*,$$ 
provided we show that $\int | Q_0 |^{\alpha+2} = \int | \widetilde{Q}_0 |^{\alpha+2}$. Indeed, the latter inequality can be seen as follows. From Lemma \ref{lem:pohoidentities} we find that $Q_0 \in H^{s_0}(\R)$ satisfies the Pohozaev idenitites
\begin{equation*}
\frac{1}{2} \int | Q_{s_0} |^2 = \frac{a_{s_0}}{\alpha+2} \int | Q_0 |^{\alpha+2}, \quad \frac{1}{2} \int (-\DD)^{\frac{s_0}{2}} Q_0 |^2 = \frac{b_{s_0}}{\alpha+2} \int | Q_0 |^{\alpha+2} .
\end{equation*}
where $a_{s_0} = \frac{\alpha}{4s_0}(2s_0-1)+1$ and $b_{s_0} = \frac{\alpha}{4s_0}$. Moreover, by assumption, the ground state $Q_0 \in H^{s_0}(\R)$ optimizes the interpolation estimate \eqref{ineq:GN}. Thus we also find that
\begin{equation*}
\int | Q_0 |^{\alpha+2} = C_{\alpha,s_0} \left ( \int | (-\DD)^{\frac{s_0}{2}} Q_0 |_2 \right) ^{\frac{\alpha}{4s_0}} \left ( \int | Q_0 |^2 \right )^{\frac{\alpha}{4s_0}(2s_0-1) +1},
\end{equation*}
with $C_{\alpha, s_0} > 0$ being the optimal constant for \eqref{ineq:GN} when $s=s_0$. Combining now the last three equations, we conclude that $\int | Q_0 |^{\alpha+2} = f(\alpha, s_0)$, for some function $f(\alpha,s_0)$  that only depends on $\alpha$ and $s_0$. By repeating the same arguments for the ground state $\widetilde{Q}_{s_0}$, we thus deduce that the equality $\int |Q_{0}|^{\alpha+2} = \int |\widetilde{Q}_{0}|^{\alpha+2}$ and hence $\lambda_* = \widetilde{\lambda}_*$.

Since $\lambda_* = \widetilde{\lambda}_*$, the uniqueness result for the limiting equation as stated in Proposition \ref{prop:global} implies that $Q_* = \widetilde{Q}_*$ as well. Next we remark that $Q_*$ has a nondegenerate linearized operator $L_+ = -\frac{d^2}{dx^2} + \lambda_* -(\alpha+1) Q_*^{\alpha}$; see, e.\,g., \cite{ChGuNaTs07}). Hence we can invoke an implicit function argument at around $(Q_*, \lambda_*)$ to construct a locally unique branch $(Q_s, \lambda_s) \in C^1((1-\delta,1]; X^\alpha \times \Rplus)$, with some $\delta > 0$ small, such that 
$$
(-\DD)^s Q_s + \lambda_s Q_s - Q_s^{\alpha+1} = 0 ,
$$
and $(Q_s, \lambda_s)$ is the unique solution for $s \in (1-\delta, 1]$ in the neighborhood $N = \{ (Q, \lambda) \in X^\alpha \times \Rplus : \| Q - Q_* \|_{X^\alpha} + | \lambda - \lambda_* | < \eps \}$, where $\eps > 0$ is a small constant. Since $Q_s \to Q_*$ and $\widetilde{Q}_s \to Q_*$ in $L^2(\R) \cap L^{\alpha+2}(\R)$ and $\lambda_s \to \lambda_*$ and $\widetilde{\lambda}_s \to \lambda_*$ both as $s \to 1$, we conclude that the branches $(Q_s, \lambda_s)$ and $(\widetilde{Q}_s, \widetilde{\lambda}_s)$ must intersect at some $s \in [s_0, 1)$. But this is a contradiction to the local uniqueness of the branches $(Q_s, \lambda_s)$ and $(\widetilde{Q}_s, \widetilde{\lambda}_s)$, as given by Proposition \ref{prop:local}. This proves uniqueness of ground states as stated in Theorem \ref{thm:unique}.

Finally, we establish uniqueness of optimizers for the Gagliardo-Nirenberg inequality \eqref{ineq:GN}. Here we simply note that, by rearrangement inequalities, we have 
\begin{equation} \label{ineq:Frank}
J^{s,\alpha}(v^*) \leq J^{s,\alpha}(v) ,
\end{equation}
where $v^* = v^*(|x|) \geq 0$ denotes the symmetric-decreasing rearrangement of $v \in H^s(\R)$. From \cite{FrSe08} we see that strict inequality holds in \eqref{ineq:Frank}, unless $v(x)$ equals $v^*(|x|)$ up to a complex phase and spatial translation. Since $v$ minimizes $J^{s,\alpha}$ and so does $v^*$, we deduce that $v^* = v^*(|x|) \geq 0$ solves the corresponding Euler-Lagrange equation
$$
(-\DD)^s v^* + \lambda v^* - \mu (v^*)^{\alpha+1} = 0,
$$ 
with some positive constants $\lambda > 0$ and $\mu > 0$. By a simple rescaling argument and uniqueness of the ground state $Q=Q(|x|) >0$, we see that $v^*(|x|) = a Q(b |x|)$ for some constants $a > 0$ and $b > 0$.

The proof of Theorem \ref{thm:unique} is now complete. \hfill $\blacksquare$

\begin{appendix}

\section{Some Uniform Bounds}

In this section, we derive some uniform bounds (with respect to $s$) for the heat kernel $e^{-t (-\DD)^s}$ with $0 < s < 1$. Moreover, as a a direct consequence, we obtain corresponding uniform bounds for the resolvent $((-\DD)^s + \lambda)^{-1}$. 

Although many of the following bounds can be directly inferred from the literature for each $0 < s <1$ individually, we were not able to find a reference, which yields the desired bounds in a uniform fashion for $s_0 \leq s < 1$ with $s_0 > 0$ fixed. Also, we mention that it is straightforward to generalize the following arguments to any space dimension. However, due to notational convenience, we have decided to focus on the one-dimensional case in what follows.   

Consider the heat kernel $e^{-t (-\DD)^s}$ on $\R$ with $0 < s < 1$. That is, we consider the Fourier transform of $e^{-t |Ê\xi |^{2s} } \in L^1(\R)$ given by 
\begin{equation}
P^{(s)}(x,t) = \frac{1}{\sqrt{2 \pi}} \int_{\R} e^{-t |\xi|^{2s}} e^{-i \xi x} \, d\xi ,
\end{equation} 
where $t > 0$ is a parameter. Note the scaling property $P^{(s)}(x,t) = t^{-\frac{1}{2s}} P^{(s)}(t^{-\frac{1}{2s}} x, 1)$ for $x \in \R$ and $t >0$. Moreover, it is obvious that $P^{(s)}(x,t)$ is an even function of $x$. We first record the following known (but not completely obvious) positivity and monotonicity result.

\begin{lemma} \label{lem:Pt_unimodal}
Let $0< s < 1$ and $t > 0$ be fixed. Then $P^{(s)}(x,t) > 0$ for $x \in \R$ and $\frac{d}{dx} P^{(s)}(x,t) < 0$ for $x > 0$.
\end{lemma}

\begin{proof}
We give the following (fairly simple) proof, which mainly rests on Bernstein's theorem about the Laplace transform.

First, by the scaling property of $P^{(s)}(x,t)$, we can assume that $t=1$ holds. Now we consider the nonnegative function $f(E) = E^s$ on the positive real line $(0,\infty)$. Using that $0 < s < 1$, it is easy to see that $f'(E)$ is completely monotone (i.\,e., we have $(-1)^n f^{(n)}(E) \leq 0$ for all $n \in \N$). This fact, in turn, implies that the map $E \mapsto e^{- f(E)}$ is completely monotone as well. Hence, by Bernstein's theorem, we infer that $e^{- f(E)} = \int_0^\infty e^{-\tau E} d \mu_{f}(\tau)$ for some nonnegative finite measure $\mu_{f}$ depending on $f$. Setting $E = |\xi|^2$ and recalling the inverse Fourier transform of the Gaussian $e^{-\tau |\xi|^2}$, we obtain the following {\em ``subordination formula''} given by
\begin{equation} \label{eq:sub}
P^{(s)}(x,1) = \int_0^\infty \frac{1}{\sqrt{2 \tau}} e^{-x^2/(4\tau)} \, d\mu_{s}(\tau) 
\end{equation}
with some nonnegative finite measure $\mu_{s} \geq 0$ and $\mu_{s} \not \equiv 0$. From this formula we readily deduce that $P^{(s)}(x,1) > 0$ for $x \in \R$ and $\frac{d}{dx} P^{(s)}(x,1) < 0$ for $x > 0$. As remarked above, this yields the desired result for all $t > 0$.\end{proof}
  
Next, we derive the following pointwise estimate for $P^{(s)}(t,x)$. 

\begin{lemma} \label{lem:feller}
For $0 < s_0 < 1$ fixed, we have the pointwise bound 
$$
P^{(s)}(x,t) \leq C \min \left \{ t^{-\frac{1}{2s}} , |x|^{-1} \right \} 
$$ 
for $x \in \R$, $t > 0$ and $s_0 \leq s < 1$. Here the constant $C > 0$ depends only on $s_0$.
\end{lemma}

\begin{remark}\label{rem:blumenthal} {\em
By a classical result in \cite{BlGe60}, we can obtain the following bound 
$$
\frac{A}{|x|^{1+2s}} \leq P^{(s)}(x,t=1) \leq \frac{B}{|x|^{1+2s} } \quad \mbox{for $|x| \geq 1$},
$$
where the constants $A > 0$ and $B > 0$ depend on $s$. However, the arguments given there do not provide any insight on how to obtain uniform decay bounds with respect to $s \geq s_0 > 1$.
}
\end{remark}

\begin{proof}
First, we easily obtain the bound
$$
P^{(s)}(x,t) \leq \int_{-\infty}^{+\infty} e^{-t |\xi|^{2s}} \,d \xi = \frac{1}{s} \Gamma \left (\frac{1}{2s} \right ) t^{-\frac{1}{2s}} \leq C t^{-\frac{1}{2s}},
$$
with some constant $C > 0$ depending only on $s_0$. Furthermore, an integration by parts yields that  $x P^{(s)}(x,t) = - i \int \left  ( \frac{d}{d \xi} e^{-t |\xi|^{2s}}  \right ) e^{-i x \xi} \, d\xi$. Hence, we find that
$$
| x P^{(s)}(x,t) | \leq \int_{-\infty}^{+\infty} 2s t |\xi|^{2s-1} e^{-t | \xi |^{2s}} \, d\xi = \int^{+\infty}_{-\infty} e^{-|u|} \, du = 2,
$$
which completes the proof. \end{proof}

Now, we consider the kernel for the resolvent $((-\DD)^s + \lambda)^{-1}$ on $\R$ with $\lambda > 0$. By functional calculus, we have the general formula
\begin{equation} \label{eq:resolvent_laplace}
\frac{1}{(-\DD)^s + \lambda} = \int_0^\infty e^{-\lambda t} e^{-t (-\DD)^s} \, dt .
\end{equation}
We have the following properties of the integral kernel associated to $((-\DD)^s + \lambda)^{-1}$.

\begin{lemma} \label{lem:resolvent_p}
Let $G_{s,\lambda} \in \mathcal{S}'(\R)$ denote the (distributional) Fourier transform of $(|\xi|^{2s} +\lambda)^{-1}$ with $0 < s < 1$ and $\lambda > 0$. Then the following properties hold.
\begin{enumerate}
\item[(i)] $G_{s,\lambda} \in L^p(\R)$ for $1 < p <  \infty$ with $1 - \frac{1}{p} < 2s$. 
\item[(ii)] $G_{s,\lambda}(x) > 0$ for $x \in \R$ and $G_{s,\lambda}(x)$ is an even function and strictly decreasing in $|x|$.
\item[(iii)] For $0 < s_0 < 1$ fixed,  we have
$$
\| G_{s,\lambda} \|_p \leq C \left ( \frac{p}{p-1} \right )^{\frac{1}{p}} \lambda^{\frac{1}{2s} (1-\frac{1}{p}) -1} \Gamma \left (  1 -\frac{1}{2s} (1-\frac{1}{p} ) \right ) .
$$
for $s_0 \leq s < 1$ and $1 < p < \infty$ with $1-\frac{1}{p} < 2s$. Here the constant $C > 0$ only depends on $s_0$. 
\item[(iv)] For $0 < s_0 < 1$ fixed, we have the pointwise bound
$$
0 < G_{s,\lambda}(x) \leq \frac{C}{\lambda |x|} \, ,
$$
for $|x| > 0$ and $s_0 \leq s < 1$, where the constant $C > 0$ only depends on $s_0$.
\end{enumerate}
\end{lemma}

\begin{proof}
As for property (i), this will clearly follow once we have deduced that (iii) holds. To see that (ii) holds, we simply recall formula \eqref{eq:resolvent_laplace} and use the corresponding properties of $P^{(s)}(t,x)$ in Lemma \ref{lem:Pt_unimodal}. To prove (iii), we note that \eqref{eq:resolvent_laplace} yields
$$
\| G_{s,\lambda} \|_p  \leq \int_0^\infty e^{-\lambda t} \| P^{(s)}(\cdot, t) \|_p \, dt .
$$
From Lemma \ref{lem:feller} we conclude that
$$
\| P^{(s)}(\cdot,t) \|_p \leq C \left ( \int_{|x| < t^{\frac{1}{2s}}} t^{-\frac{p}{2s}} \, dx + \int_{|x| \geq t^{\frac{1}{2s}}} |x|^{-p} \, dx \right )^{\frac{1}{p}} \leq C \left ( \frac{p}{p-1} \right )^{\frac{1}{p}} t^{-\frac{1}{2s} (1 - \frac{1}{p} ) },
$$
with $1 < p < \infty$ and where $C > 0$ only depends on $s_0$. A straightforward combination of these bounds yields the desired estimate, provided that $1-\frac{1}{p} < 2s$ holds.

To establish the pointwise bound stated in (iv), we simply use \eqref{eq:resolvent_laplace} in combination with Lemma \ref{lem:feller}.  \end{proof}

We conclude this section by deriving a uniform bound for the optimal constants $C_{\alpha,s}> 0$ for the Gagliardo-Nirenberg inequality (in $d=1$ dimensions)
\begin{equation*}
\int | f |^{\alpha+2} \leq C_{\alpha,s} \left (  \int |(-\Delta)^{\frac s 2} f |^2 \right )^{\frac{\alpha}{4s}} \cdot \left ( \int | f |^2 \right )^{\frac{\alpha}{4s} (2s - 1) + 1},
\end{equation*}
where $0 <  s < 1$ and $0 < \alpha < \amax(s)$. We have the following uniform bound.

\begin{lemma} \label{lem:GNuniform}
Let $0<\alpha<\infty$ be given. Then there is a constant $K_\alpha > 0$ such that $C_{\alpha,s}\leq K_\alpha$ for $\frac{\alpha}{2(\alpha+2)} \leq s < 1$. 
\end{lemma}

\begin{proof}
Let $s_0=\frac{\alpha}{2(\alpha+2)}$ and note that $0 < s_0 <1/2$. By Sobolev inequalities, we have $\| f \|_{\alpha+2} \leq \widetilde{K} \| (-\DD)^{\frac{s_0}{2}} f \|_2$ for some constant $\widetilde{K}> 0$ depending only on $\alpha$. Next, we use that $\rho H \leq \theta H^{1/\theta} + (1-\theta) \rho^{1/(1-\theta)}$ for any nonnegative operator $H \geq 0$ and any real numbers $\rho>0$ and $0<\theta<1$. Evaluating this operator inequality on a function $f$ and optimizing with respect to $\rho$, we find that
$$
\langle f,Hf \rangle \leq \langle f,H^{1/\theta}fÊ\rangle^{\theta} \|f\|_2^{2(1-\theta)} \,.
$$
Given $s>s_0$, we apply this to $H=(-\Delta)^{s_0}$ with $\theta=s_0/s$. This  yields
$$
\| f \|_{\alpha+2} \leq \widetilde{K} \| (-\DD)^{\frac{s_0}{2}} f \|_2 \leq \widetilde{K} \| (-\DD)^{\frac{s}{2}} f \|_2^{s_0/s} \| f \|_2^{1-s_0/s} \,,
$$
whence the result follows with $K_\alpha = \widetilde{K}^{\alpha+2}$. \end{proof}

\section{Regularity, Symmetry and Monotonicity}

\label{app:regularity}

Let $0 < s < 1$ and $0 < \alpha < \amax(s)$ be fixed throughout this section. We consider (not necessarily real-valued) solutions $Q=Q(x)$ in the distributional sense of the equation
\begin{equation} \label{eq:Qappendix}
(-\DD)^s Q + \lambda Q - |Q|^{\alpha} Q = 0 ,
\end{equation}
where $\lambda > 0$ is given. Again, we could assume that $\lambda=1$ by a rescaling argument, but we keep $\lambda >0$ explicit in the following.

We start with a simple regularity result used in Section \ref{sec:unique}.

\begin{lemma}
If $Q \in L^2(\R) \cap L^{\alpha+2}(\R)$ solves \eqref{eq:Qappendix}, then $Q \in H^s(\R)$. 
\end{lemma}

\begin{remcom} {\em
Formally, this regularity result follows from integrating \eqref{eq:Qappendix} against $Q$. However, this argument is not legitimate, since we only assume that $(-\DD)^s Q \in H^{-2s}(\R)$ a-priori.}
\end{remcom}

\begin{proof}
Using that $Q = ((-\DD)^s + \lambda)^{-1} |Q|^\alpha Q$ for $Q \in L^2(\R) \cap L^{\alpha+2}(\R)$,  we deduce that
$$
\| (-\DD)^\frac{s}{2} Q \|_2 = \left \| \frac{(-\DD)^{\frac{s}{2}}}{(-\DD)^s + \lambda} |Q|^\alpha Q \right \|_2 \lesssim_{\lambda} \left \| \frac{1}{(-\DD)^{ \frac{s}{2} } + 1} |Q|^\alpha Q \right \|_2 .
$$
Now, we invoke Lemma \ref{lem:resolvent_p}, part (iii), with $s_0 = s/2$ and use Young's inequality. Indeed, since $1+\frac{1}{2} = \frac{1}{p}  + \frac{\alpha+1}{\alpha+2}$ implies that $1-\frac{1}{p} = \frac{\alpha}{\alpha+2} < 2s$ since $\alpha < \amax(s)$, we deduce
$$
\left \| \frac{1}{(-\DD)^{\frac{s}{2}} + 1} |Q|^\alpha Q \right \|_2 \lesssim_{s} \| |Q|^\alpha Q \|_{\frac{\alpha+2}{\alpha+1}} = \| Q \|_{\alpha+2}^{\alpha+1}.
$$
Therefore $Q \in L^2(\R)$ satisfies $\| (-\DD)^\frac{s}{2} Q \|_2 < \infty$ and hence $Q \in H^s(\R)$. \end{proof}

\noindent
Next, we proceed with the following improved regularity result.

\begin{lemma}  \label{lem:Qregular}
If $Q \in H^s(\R)$ solves \eqref{eq:Qappendix}, then $Q \in H^{2s+1}(\R)$.
\end{lemma}

\begin{remark} {\em
If $\alpha=1,2,\ldots$ is an integer in equation \eqref{eq:Qappendix}, it is easy to see that $Q \in H^k(\R)$ for all $k \geq 1$. See also \cite{LiBo96} for an analyticity result of $Q(x)$ in this case.}
\end{remark}

\begin{proof}
 First, we remark that $Q \in L^\infty(\R)$ holds. Of course, this fact immediately follows if $s > 1/2$ due to Sobolev inequalities. To see that $Q \in L^\infty(\R)$ also when $0 < s \leq 1/2$, we can use the $L^p$-bounds for the resolvent $((-\DD)^s + \lambda)^{-1}$ derived in Lemma \ref{lem:resolvent_p}. Then by iterating the identity $Q = ((-\DD)^s + \lambda)^{-1} |Q|^\alpha Q$ sufficiently many times, we conclude that $\| Q \|_{\infty} < \infty$ holds. (Alternatively, we could use that $Q^\alpha \in K_s$ and use the remarks in Section \ref{sec:nodal} to infer that $Q \in L^\infty(\R)$ holds.)

Given that $Q \in L^\infty(\R)$, we can now show that $Q \in H^{2s+1}(\R)$ as follows. Since $Q \in L^2(\R)$, it remains to derive the bound $\| (-\DD)^{s+\frac{1}{2}} Q \|_2 < \infty$. We treat the cases $s \geq 1/2$ and $0 < s < 1/2$ separately as follows. 

\subsubsection*{Case: $s \geq 1/2$}  As usual, this case is straightforward to handle. Indeed, we notice that 
\begin{align*}
\| (-\DD)^{s} Q \|_2  & = \left \| \frac{(-\DD)^{s}}{(-\DD)^s + \lambda} |Q|^{\alpha} Q \right \|_2 \lesssim_{\lambda} \left \| |Q|^\alpha Q \right \|_2 \lesssim_{\lambda} \| Q \|_\infty^\alpha \|Q \|_2 < \infty .
\end{align*}
Hence we have $Q \in H^{2s}(\R)$ and in particular $Q \in H^1(\R)$, since $s \geq 1/2$ by assumption. Next, we proceed to find that
\begin{align*}
\left \| (-\DD)^{s+\frac{1}{2}} Q \right \|_2 & =  \left \| \frac{(-\DD)^{s+\frac{1}{2}}}{(-\DD)^s + \lambda} |Q|^{\alpha} Q \right \|_2  \lesssim_{\lambda,\alpha} | \nabla (|Q|^\alpha Q) \|_2 \lesssim_{\lambda} \| Q \|^\alpha_\infty \| \nabla Q \|_2 < \infty,
\end{align*}
where we used that $| \nabla ( |Q|^\alpha Q)| \leq (\alpha+1) |Q|^\alpha |\nabla Q|$ a.\,e.~in $\R$. Thus we have shown that $Q \in H^{2s+1}(\R)$, provided that $s \geq 1/2$ holds.

\subsubsection*{Case: $0 < s < 1/2$} First, we recall that the well-known identity 
$$
\big \| (-\DD)^{\frac{\sigma}{2}} u \big \|_2^2 = \frac{2^{2 \sigma-1}}{\pi^{\frac{1}{2}}}  \frac{  \Gamma((1+2\sigma)/2) }{|\Gamma(-\sigma)|} \int \! \! \int_{\R \times \R} \frac{|u(x)-u(y)|^2}{|x-y|^{1+2\sigma}} \, dx \, dy ,
$$
for any $0 < \sigma < 1$. From this we conclude that
\begin{equation} \label{ineq:fractional}
\| (-\DD)^{\frac{\sigma}{2}} (|Q|^\alpha Q) \|_2 \lesssim_{\sigma,\alpha} \| Q \|_{\infty}^\alpha \| (-\DD)^{\frac{\sigma}{2}} Q \|_2,
\end{equation}
where we use the pointwise inequality $$| |Q|^\alpha(x) Q(x) - |Q|^\alpha(y) Q(y)| \leq \alpha \max \{ |Q|^{\alpha}(x), |Q|^{\alpha}(y) \}|Q(x)-Q(y)|.$$

Recall that $0 < s < 1/2$ by assumption, and let $N  \geq 2$ be the unique integer such that $1/(N+1) \leq s < 1/N$. By using estimate \eqref{ineq:fractional} and $Q \in L^\infty(\R)$, we conclude that
\begin{align*}
\| (-\DD)^{\frac{(k+1) s}{2}} Q \|_2 & = \left \| \frac{(-\DD)^{\frac{s}{2}} (-\DD)^{\frac{k s}{2}}}{(-\DD)^s + \lambda} |Q|^\alpha Q \right \|_2   \lesssim_{k,s,\lambda,\alpha} \| (-\DD)^{\frac{k s}{2}} (|Q|^\alpha Q) \|_2 \\
& \lesssim_{k,s,\lambda,\alpha}  \| Q \|_\infty^\alpha \| (-\DD)^{\frac{k s}{2}} Q \|_2 \lesssim_{k,s,\lambda,\alpha} \| (-\DD)^{\frac{k s}{2}} Q \|_2 ,
\end{align*}
for $k=1, \ldots, N$. By iteration and since $Q \in L^2(\R)$, we thus obtain 
$$
\| Q_s \|_{H^{(N+1) s}} \lesssim_{s,k,\alpha} \| Q_s \|_{H^{s}} < \infty.
$$
Since $(N+1) s \geq 1$, we deduce that $Q \in H^1(\R)$ holds. Given this fact, we can now conclude that $Q \in H^{2s+1}(\R)$ in the same fashion as done above for $s \geq 1/2$. 

The proof of Lemma \ref{lem:Qregular} is now complete.\end{proof}

Next, we turn to symmetry and monotonicity results about solutions for \eqref{eq:Qappendix}. Indeed, by adapting the recent moving plane arguments developed by L.~Ma and L.~Zhao in \cite{MaZh10} for the nonlocal Pekar-Choquard equation, we can derive the following symmetry and monotonicity result.

\begin{lemma}  \label{lem:symm}
If $Q \in H^s(\R)$ with $Q \geq 0$ and $Q \not \equiv 0$ solves \eqref{eq:Qappendix}, then we have
$$
Q(x)= \widetilde{Q}(|x-x_0|)
$$
with some $x_0 \in \R$ and the function $\widetilde{Q}(r)$ satisfies $\widetilde{Q}(r) > 0$ and $\widetilde{Q}'(r) < 0$ for $r > 0$.
\end{lemma}

\begin{proof}
To deduce Lemma \ref{lem:symm} with the slightly weaker statement that $\widetilde{Q}(r)$ is (not necessarily strictly)  decreasing, we can directly apply the moving plane arguments developed in \cite{MaZh10}. More precisely, by following \cite[Section 5]{MaZh10}, we only have to verify that the kernel $K=K(x-y)$ for the resolvent $((-\DD)^s+1)^{-1}$ on $\R$ satisfies the following conditions: 1.) $K(|z|)$ is real-valued and even, 2.) $K(|z|) > 0$ for $z \in \R$, and 3.) $K(|z|)$ is monotone decreasing in $|z|$. Indeed, we have all these facts about $K(x-y)$ thanks to Lemma \ref{lem:resolvent_p}, which is based on the properties of the heat kernel $e^{-t(-\DD)^s}$ on $\R$.  

Finally, we show that $\widetilde{Q}'(r) < 0$ for $r > 0$. Without loss of generality, we can assume that $x_0 = 0$ and hence $Q(x) = \widetilde{Q}(|x|) > 0$. By differentiating \eqref{eq:Qappendix} with respect to $x$, we obtain $$L_+ Q' = 0$$ 
where $L_+ = (-\DD)^s + 1 - (\alpha+1) Q^\alpha$. Note that $Q' \in L^2_{\mathrm{odd}}(\R)$ and $Q'(x) = \widetilde{Q}'(r) \leq 0$ for $x=r > 0$, since $\widetilde{Q}(r)$ is monotone decreasing. In view of Lemma \ref{lem:perron2} applied to $L_+$, we deduce that $Q' \in L^2_{\mathrm{odd}}(\R)$ is the ground state eigenfunction of $L_+$ restricted to $L^2_{\mathrm{odd}}(\R)$. Thus we either have $Q'(x) < 0$ or $Q'(x) > 0$ for $x > 0$, where the latter alternative is clearly ruled out. Hence $Q'(x) = \widetilde{Q}'(r) < 0$ for $x=r > 0$.
\end{proof}

\section{The  Kato Class $K_s$ and \\ Perron-Frobenius Theory for $H=(-\DD)^s+V$}

\label{app:kato}

In this section, we collect some basic results about fractional Schr\"odinger operators
$$
H=(-\DD)^s + V  \quad \mbox{acting on $L^2(\R)$}.
$$
Although most of our discussion generalizes to higher space dimensions $d \geq 1$, we shall content ourselves with the one-dimensional case. In Section \ref{sec:nodal}, we defined $K_s$ as the {\em `Kato-class'} with respect to $(-\DD)^s$; see Definition \ref{def:kato}. In particular, the condition $V \in K_s$ guarantees that the heat semi-group $e^{-t H}$ maps to $L^2(\R)$ to $L^\infty(\R) \cap C^0(\R)$ for $t > 0$. In particular, any $L^2$-eigenfunction of $H$ is bounded and continuous. See \cite{CaMaSi90} for more details.

First, we derive the following sufficient condition in terms of $L^p$-spaces for a potential $V$ to be in $K_s$. (Although the following result may be known in the literature, we were not able to find a suitable reference.)

\begin{lemma} \label{lem:katoclass}
Let $0 < s < 1$ and $V : \R \to \R$ be given. Then the following holds.

 If $0 < s \leq 1/2$ and $V \in L^p(\R)$ for some $p > 1/2s$, then $V \in K_s$. If $1/2 < s < 1$ and $V \in L^p(\R)$ for some $p \geq 1$, then $V \in K_s$.
\end{lemma}

\begin{proof}
In view of Definition \ref{def:kato}, we have to show that 
\begin{equation} \label{eq:kato_limit}
\lim_{E \to +\infty}  \left \| ((-\DD)^s + E)^{-1} |V| Ê\right \|_{L^\infty \to L^\infty} = 0.
\end{equation}
Using that $((-\DD)^s +E)^{-1} = \int_0^\infty e^{-Et} e^{-t(-\DD)^s} dt$ for $E > 0$ and H\"older's inequality, we obtain
$$
\left \|( (-\DD)^s+E)^{-1} |V| \right \|_{L^\infty \to L^\infty}  \leq \| V \|_p \int_0^\infty e^{-Et} \| e^{-t (-\DD)^s} \|_{L^p \to L^\infty} \, dt
$$
Next, let $P^{(s)}(x,t)$ denote the kernel of $e^{-t(-\DD)^s}$ on $\R$. By Young's inequality, we have $\| e^{-t(-\DD)^s} \|_{L^p \to L^\infty} \leq \| P^{(s)}(\cdot, t) \|_{q}$ with $\frac{1}{p} + \frac{1}{q} = 1$. Next, we find that
$$
\| P^{(s)}(\cdot,t) \|_q \leq C\left ( \int_{|x| < t^{\frac{1}{2s}}} t^{-\frac{q}{2s}} \, dx +   \int_{ |x| \geq t^{\frac{1}{2s}}} \frac{t^q}{|x|^{q(1+2s)}} \, dx \right )^\frac{1}{q} \leq C t^{-\frac{1}{2s} \frac{q-1}{q} },
$$
where the constant $C > 0$ only depends on $s$.  Indeed, this follows from the simple bound $ P^{(s)}(x,t) \leq  C t^{-\frac{1}{2s}}$ for all $x \in \R$ from Lemma \ref{lem:feller}, combined with $s$-dependent bound stated in Remark \ref{rem:blumenthal} and the scaling property $P^{(s)}(x,t) = t^{-\frac{1}{2s}} P^{(s)}(t^{-\frac{1}{2s}} x,1)$ for $t > 0$.
Because of $\frac{1}{p}=\frac{q-1}{q}$, the previous bound for $\| P^{(s)}(\cdot,t) \|_q$ implies that
$$
\left \|( (-\DD)^s+E)^{-1} |V| \right \|_{L^\infty \to L^\infty} \leq C \| V \|_p \int_0^\infty e^{-Et} t^{-\frac{1}{2sp}  } \, dt  .
$$
From this we deduce that \eqref{eq:kato_limit} holds if $p>1/2s$ for $s \leq 1/2$, or if $p \geq 1$ for $s > 1/2$. \end{proof}

As a next result, we show that fractional Schr\"odinger operators $H= (-\DD)^s +V$ enjoy the following `Perron-Frobenius' property.

\begin{lemma} \label{lem:perron}
Let $0 < s <1$ and consider $H = (-\DD)^s + V$ acting on $L^2(\R)$, where we assume that $V \in K_s$. Suppose that $e = \inf \sigma(H)$ is an eigenvalue. Then $e$ is simple and its corresponding eigenfunction $\psi = \psi(x) > 0$ is positive (after replacing $\psi$ by $-\psi$ if neccessary).
\end{lemma}

\begin{proof}
By Lemma \ref{lem:Pt_unimodal}, the operator $e^{-t (-\DD)^s}$ acting on $L^2(\R)$ is {\em positivity improving} for $t > 0$. By this, we mean that if $f \geq 0$ and $f \not \equiv 0$, then $e^{-t (-\DD)^s} f > 0$.

Next, we consider $H=(-\DD)^s+V$ acting on $L^2(\R)$. Since $V \in K_s$, it follows that $V$ is an infinitesimally bounded perturbation of $(-\DD)^s$. Hence we can apply standard Perron-Frobenius type arguments (see, e.\,g., \cite{ReSi78}) to deduce that the largest eigenvalue of $e^{-tH}$ is simple and its corresponding eigenfunction strictly positive. By functional calculus, this fact is equivalent to saying that the lowest eigenvalue of $H$ is simple and has a positive eigenfunction. 
\end{proof}

\begin{lemma} \label{lem:perron2}
Let $H = (-\DD)^s + V$ be as in Lemma \ref{lem:perron}. Moreover, we assume that $V = V(|x|)$ is even and let $H_{\mathrm{odd}}$ denote the restriction of $H$ to $L^2_{\mathrm{odd}}(\R)$. If $e = \inf \sigma ( H_{\mathrm{odd}})$ is an eigenvalue, then $e$ is simple and the corresponding odd eigenfunction $\psi = \psi(x)$ satisfies $\psi(x) > 0$ for $x > 0$ (after replacing $\psi$ by $-\psi$ if neccessary).
\end{lemma}

\begin{proof}
This result follows by a slight twist of standard abstract Perron-Frobenius arguments.

Let $(-\DD)^s_{\mathrm{odd}}$ denote the restriction of $(-\DD)^s$ on $L^2_{\mathrm{odd}}(\R)$. By odd symmetry, we find that $e^{-t (-\DD)^s_{\mathrm{odd}}}$  acts on $f \in L^2_{\mathrm{odd}}(\R)$ according to
 \begin{equation}
(e^{-t (-\DD)^s_{\mathrm{odd}}} f)(x) = \int_0^\infty K_{t,s}(x,y) f(y) \, dy .
\end{equation}
Here the integral kernel $K_{t,s}(x,y)$ is given by
\begin{equation}
K_{t,s}(x,y) = P^{(s)}(x-y,t) - P^{(s)}(x+y,t),
\end{equation}
with $P^{(s)}(x,t)$ denoting the Fourier transform of $e^{-t |\xi|^{2s}}$ in $\R$. Now, we claim that $K_t(x,y) > 0$ holds for $0 < x,y < \infty$. Indeed, recall that $P^{(s)}(x,t)$ is even in $x$, positive and strictly decreasing with respect to $|x|$; see Lemma \ref{lem:Pt_unimodal}. Hence if we write $z = x-y$ and $z' = x+y$ for $x,y > 0$, we easily check that $|z| < |z'|$ holds. Therefore we deduce that $K_{t,s}(x,y) > 0$ is a strictly positive kernel on $L^2(\Rplus)$. Hence $e^{-t (-\DD)^s_{\mathrm{odd}}}$ can be identified with a positivity improving operator on $L^2(\Rplus)$. 	

Now, we consider $H_{\mathrm{odd}} = (-\DD)^s_{\mathrm{odd}} + V$ with $V = V(|x|)$ even. Using standard Perron-Frobenius arguments (see the proof of Lemma \ref{lem:perron} and reference there), we deduce that the largest eigenvalue of $e^{-tH_{\mathrm{odd}}}$ on  $L^2(\Rplus)$ is simple and its corresponding eigenfunction satisfies $\psi_0 = \psi_0(x) > 0$ for $x > 0$. By functional calculus, this fact now implies Lemma \ref{lem:perron2} about $H_{\mathrm{odd}}$. \end{proof}

\section{A Topological Lemma}

The following auxiliary result is needed in Section \ref{sec:nodal}.

\begin{lemma} \label{lem:jordan}
Let $x_1 < x_2 < x_3 < x_4$ be real numbers. Suppose that $\gamma, \widetilde{\gamma} : [0,1] \to \overline{\R}^+_2$ are simple (i.\,e.~injective) continuous curves such that 
\begin{enumerate}
\item[(i)] \mbox{$\gamma(0) = (x_1,0), \gamma(1) = (x_3,0)$ and $\gamma(t) \in \R^2_+$ for $t \in (0,1)$}.
\item[(ii)] \mbox{$\widetilde{\gamma}(0)=(x_2,0), \gamma(1) = (x_4,0)$ and $\gamma(t) \in \R^2_+$ for $t \in (0,1)$}. 
\end{enumerate}
Then $\gamma$ and $\widetilde{\gamma}$ intersect in $\R^2_+$, i.\,e., we have $\gamma(t) = \widetilde{\gamma}(t_*)$ for some $t \in (0,1)$ and $t_* \in (0,1)$. 
\end{lemma}

\begin{proof}
We define the continuous curve $\widehat{\gamma} : [0,1] \to \overline{\R}_+^2$ by setting 
$$
\widehat{\gamma}(t) := \left \{ \begin{array}{ll} \gamma(2t) & \mbox{for $0 \leq t \leq 1/2$,}\\
( (2t-1)(x_1 - x_3) +x_3, 0) & \mbox{for $1/2 < t \leq 1$.} \end{array} \right .
$$ 
Note that $\widehat{\gamma}(0) = \widehat{\gamma}(1) = (x_1,0)$. Clearly $\widehat{\gamma}$ is a Jordan curve (i.\,e., a simple and closed continuous curve) in $\R^2$. By Jordan's curve theorem (see \cite{Ma84} for a simple proof based on Brouwer's fixed point theorem) the set $A = \R^2 \setminus  \widehat{\gamma}([0,1])$ has exactly two open connected components in $\R^2$. Let us denote these two components by $B$ and $C$ in what follows. Moreover, we have that $B$, say, is bounded, whereas the component $C$ is unbounded. Finally, the Jordan curve theorem states that $\widehat{\gamma}([0,1]) = \partial B = \partial C$ holds. Next, we consider the sets 
$$N_{\eps,+}(x_2) = \{ (x,y) \in \R^2 : \sqrt{(x-x_2)^2 + y^2} < \eps, \; y > 0 \},
$$
$$
N_{\eps,-}(x_2) = \{ (x,y) \in \R^2 : \sqrt{(x-x_2)^2 + y^2} < \eps, \; y < 0 \},
$$
where $\eps > 0$ is given. Since $x_1 < x_2 < x_3$ by assumption and by construction of $\widehat{\gamma}$, we have that $(x_2,0) \in \widehat{\gamma}([0,1])$. Suppose now that $(\tilde{x}, \tilde{y}) \in N_{\eps,-}(x_2)$ where $\eps > 0$ is arbitrary. Clearly, we can connect the point $(\tilde{x}, \tilde{y})$ with $(x_4,0)$ by a continuous curve in the lower halfplane without intersecting the Jordan curve $\widehat{\gamma}$. Furthermore, it is obvious $(x_4,0)$ that belongs to the unbounded component $C$ (by connecting it to $(x_4,y)$ with $y \to -\infty$ without intersecting $\widehat{\gamma}$.) Hence we conclude that $N_{\eps,-}(x_2) \subset C$ for any $\eps > 0$. On the other hand, we recall that $\partial A = \widehat{\gamma}([0,1])$. Since $N_{\eps,-}(x_2) \cap B = \emptyset$ for all $\eps > 0$, we find that $N_{+,\eps}(x_2) \subset B$ for some $\eps > 0$ sufficiently small. 

Now we conclude as follows. First, we note that $N_{\eps,+}(x_4) \subset C$ for $\eps > 0$ sufficiently small, since $C$ is open and $(x_4,0) \in C$. Second, from $\widetilde{\gamma}(0) = (x_2,0)$ and $\widetilde{\gamma}(1) = (x_4,0)$ and by continuity, we deduce from $\widetilde{\gamma}(t) \in \R^2_+$ for $t \in (0,1)$ that
$$
\mbox{$\widetilde{\gamma}(t) \in N_{+,\eps}(x_2) \subset B$ for $t$ close to 0}, \quad \mbox{$\widetilde{\gamma}(t) \in N_{\eps,+}(x_4) \subset C$ for $t$ close to 1}.
$$
with some $\eps > 0$ sufficiently small.  Hence there exists  $t_* \in (0,1)$ such that $\widetilde{\gamma}(t_*) \in \widehat{\gamma}([0,1])$. But since $\widetilde{\gamma}(t_*)$ lies in the upper halfplane $\R^2_+$, we actually deduce that $\widetilde{\gamma}$ must intersect $\gamma$ in $\R^2_+$. \end{proof}

\section{Regularity of $F$}

We define the map
\begin{equation}
F(Q,\lambda, s) := \left [ \begin{array}{c} \displaystyle Q- \frac{1}{(-\DD)^s + \lambda}  |Q|^\alpha Q  \\[1ex] \| Q  \|_{\alpha+2}^{\alpha+2} - c_0 \end{array} \right ] ,
\end{equation}
for $Q \in L^2(\R) \cap L^{\alpha+2}(\R)$, $s_0 \leq s < 1$ and $\lambda > 0$. Here $c_0 \in \R$ is some fixed constant. 

\begin{lemma} \label{lem:F_C1}
Let $0 < s_0 < 1$ and $0 < \alpha < \amax(s_0)$ be fixed. Consider the real Banach space $X^{\alpha} = L^2(\R) \cap L^{\alpha+2}(\R)$ equipped with the norm $\| \cdot \|_{X^\alpha} = \| \cdot \|_2 +  \| \cdot \|_{\alpha+2}$. Define $F(Q,\lambda,s)$ as above. Then the map $F : X^\alpha \times \Rplus \times [s_0,1) \to X^\alpha \times \R$ is $C^1$.
\end{lemma}

\begin{proof}
First, we show that $F : X^\alpha \times \Rplus \times [s_0,1) \to X^\alpha \times \R$ is well-defined. From Lemma \ref{lem:resolvent_p} together with Young's and H\"older's inequality we find that
\begin{equation} \label{ineq:Fwell}
\left \| \frac{1}{(-\DD)^s + \lambda} |Q|^\alpha f \right \|_{q} \lesssim_{\lambda, s, p} \| Q \|_{r}^{\alpha} \| f \|_r,
\end{equation}
where $1 < p < \infty$ and $1 \leq q,r \leq \infty$ satisfy 
\begin{equation}
\frac{1}{q} + 1 - \frac{1}{p} = \frac{\alpha+1}{r}, \quad 1- \frac{1}{p} < 2s.   
\end{equation}
In particular, if we choose $r=\alpha+2$ and $q=2$, we find $1-\frac{1}{p} = \frac{\alpha}{2 (\alpha+2)} < s_0 < 2s$ since $\alpha < \amax(s_0)$. Furthermore, by setting $r = \alpha+2$ and $q=\alpha+2$, we see that $1-\frac{1}{p} = \frac{\alpha}{\alpha+2} < 2 s_0 \leq 2s$ due to $\alpha < \amax(s_0)$. Hence we can apply \eqref{ineq:Fwell} to conclude that $F(Q,\lambda,s)$ is well-defined. 

Next, we turn to the Fr\'echet differentiability of $F$. (Recall that we restrict to real-valued functions.) First, we consider the second component of the map $F=(F_1, F_2)$, which is given by
$$
F_2(Q,\lambda,s) := \| Q \|_{\alpha+2}^{\alpha+2} -c_0 .
$$
with some fixed constant $c_0 \in \R$.  It is easy to see that $F_2(Q, \lambda,s)$ is Fr\'echet differentiable with 
$$
\frac{\partial F_2}{\partial Q} = (\alpha+2) \langle |Q|^{\alpha} Q, \cdot \rangle, \quad \frac{\partial F_2}{\partial \lambda} =0, \quad \frac{\partial F_2}{\partial s} = 0,
$$
where $\langle f, \cdot \rangle$ denotes the map $g \mapsto \langle f, g\rangle$. Moreover, it is straightforward to check that $\frac{\partial F_2}{\partial Q}$ depends continuously on $Q$ with respect to the topology in $X = L^2(\R) \cap L^{\alpha+2}(\R)$. Let us now turn to the Fr\'echet differentiability of the first component 
$$
F_1(Q, \lambda, s) := Q - \frac{1}{(-\DD)^s + \lambda} |Q|^{\alpha} Q .
$$
We claim that
$$
\frac{\partial F_1}{\partial Q} = 1 - \frac{1}{(-\DD)^s + \lambda} (\alpha+1) |Q|^{\alpha} , \quad \frac{\partial F_1}{\partial \lambda} = \frac{1}{((-\DD)^s + \lambda)^2} |Q|^{\alpha} Q, 
$$
and
$$
\frac{\partial F_1}{\partial s} = \frac{(-\DD)^s \log (-\DD) }{((-\DD)^s + \lambda)^2} |Q|^{\alpha} Q .  
$$
Indeed, it follows from standard arguments (e.\,g,~Sobolev embeddings, H\"older inequality) combined with \eqref{ineq:Fwell}) that the derivatives $\frac{\partial F_1}{\partial Q}$, $\frac{\partial F_1}{\partial \lambda}$ and $\frac{\partial F_1}{\partial s}$ exist and are given as above. For instance, to prove this claim for $\frac{\partial F_1}{\partial s}$ we argue as follows. Let $(Q,\lambda,s) \in X \times \Rplus \times [s_0,1)$ be fixed and suppose that $s+h \in [s_0, 1)$ with $h \in \R$ and $h \neq 0$. We have to show that
\begin{equation*}
F_1(Q, \lambda, s+h) - F_1(Q, \lambda, s) = \frac{\partial F_1}{\partial s}(Q,\lambda,s) h + r(h) ,
\end{equation*}
where $|h|^{-1} r(h) \to 0$ in $X^\alpha$ as $h \to 0$. To show this fact, we consider the function 
$$
f(\xi,s) :=  \frac{1}{{|\xi|^{2s}+ \lambda}} \quad \mbox{for $\xi \in \R$ and $s \in [s_0,1)$.}
$$ 
An elementary calculation yields
$$
\frac{\partial f}{\partial s} = - \frac{|\xi|^{2s} \log (|\xi|^2) }{(|\xi|^{2s} + \lambda)^2}, \quad
\frac{\partial^2 f}{\partial s^2} = 2 \frac{ |\xi|^{4s}  (\log(|\xi|^2))^2}{(|\xi|^{2s} + \lambda)^3} - \frac{ |\xi|^{2s}  (\log (|\xi|^2))^2}{(|\xi|^{2s} + \lambda)^2} .
$$
In particular, for any $s_0/2 > \sigma > 0$ and $s \geq s_0$, we have the following bounds
$$
\left | \frac{\partial f}{\partial s } \right | \lesssim_{s_0, \sigma} \frac{|\xi|^{2s+\sigma} +1}{(|\xi|^{2s} + \lambda)^2} \lesssim_{\sigma,s_0,\lambda}  \frac{1}{|\xi|^{2s_0-\sigma} +1 },
$$
$$
\left | \frac{\partial^2 f}{\partial s^2 } \right | \lesssim_{s_0,\sigma} \frac{|\xi|^{4s+\sigma} +1}{(|\xi|^{2s} + \lambda)^3} + \frac{|\xi|^{2s+\sigma} +1}{(|\xi|^{2s} + \lambda)^2}  \lesssim_{\delta,s_0,\sigma} \frac{1}{|\xi|^{2s_0-\sigma} +1 } .
$$
Next, by Sobolev inequalities, we obtain
\begin{equation} \label{ineq:sobX}
\| u \|_{X^\alpha} \lesssim_{\alpha} \| ((-\DD)^{\frac{s_\alpha}{2}} + 1) u \|_2 \quad \mbox{with $s_\alpha = \displaystyle \frac{\alpha}{2(\alpha+2)}$.}
\end{equation}
Note that $s_\alpha < s_0$ since $\alpha < \amax(s_0)$. Now, by Plancherel's identity and Taylor's theorem applied to $f(\xi,s)$ and estimate \eqref{ineq:Fwell}, we deduce (with $\frac{\partial F_1}{\partial s}$ given above) the following estimate:
\begin{align*}
& \left \| F_1(Q, \lambda,s+h) - F_1(Q,\lambda,s) - \frac{\partial F_1}{\partial s} (Q,\lambda,s) h \right \|_{X^\alpha} \\
& \lesssim_{\alpha}  h^2 \sup_{\xi \in \R} \left | (| \xi|^{s_\alpha} +1) \frac{\partial^2 f}{\partial s^2}(\xi,s) (|\xi|^{s_\alpha+\eps}+1) \right | \left \| \frac{1}{(-\DD)^{\frac{s_\alpha+\eps}{2}} + 1} |Q|^\alpha Q \right \|_2 \\ 
&\lesssim_{\alpha, \lambda, s_0, \sigma} h^2 \sup_{\xi \in \R} \left ( \frac{|\xi|^{2s_\alpha + \sigma} +1}{|\xi|^{2s_0- \sigma} + 1 } \right ) \| Q \|_{\alpha+2}^{\alpha+1} \lesssim_{\alpha, \lambda, s_0, \sigma} h^2 \| Q \|_{X^\alpha}^{\alpha+1} ,
\end{align*}
with some small constant $\sigma > 0$ such that  $s_0 > s_\alpha + 2 \sigma$ holds, which is possible since $s_\alpha < s_0$. Also, we used above  that $\| (-\DD)^{\frac{s_\alpha + \eps}{2}} +1)^{-1} |Q|^\alpha |Q| \|_2 \lesssim_{s_\alpha, \eps} \| Q \|_{\alpha+2}^{\alpha+1}$ by \eqref{ineq:Fwell}. Thus we conclude that $\frac{\partial F_1}{\partial s}$ exists and is given as claimed. 

Let us now turn to the continuity of $\frac{\partial F_1}{\partial Q}$, $\frac{\partial F_1}{\partial \lambda}$ and $\frac{\partial F_1}{\partial s}$. Again, this follows from standard arguments  in combination with \eqref{ineq:Fwell}. For example, to show that $\frac{\partial F_1}{\partial Q}$ depends continuously on $(Q,\lambda, s)$, we can argue as follows. Let $(Q, \lambda,s) \in X^\alpha \times \Rplus \times [s_0,1)$ be fixed and suppose that $\eps > 0$ is given. We have to find $\delta > 0$ such that
\begin{equation} \label{ineq:cont_FQ}
\left \| \Big ( \frac{\partial F_1}{\partial Q} (Q, \lambda, s) - \frac{\partial F_1}{\partial Q}(\tilde{Q}, \tilde{\lambda}, \tilde{s}) \Big ) f  \right \|_{X^\alpha} \leq \eps \| f \|_{X^\alpha},
\end{equation} 
whenever $\| Q - \tilde{Q} \|_{X^\alpha} + |\lambda-\tilde{\lambda}| + |s - \tilde{s}| \leq \delta$ and $(\tilde{Q}, \tilde{\lambda}, \tilde{s}) \in X^\alpha \times \Rplus \times [s_0,1)$. Indeed, by using \eqref{ineq:sobX}, we see that \eqref{ineq:cont_FQ} follows if we can show that
\begin{equation} \label{ineq:cont_FQ2}
\left \| ( A_{s, \lambda} |Q|^\alpha - A_{\tilde{s}, \tilde{\lambda}} |\tilde{Q}|^{\alpha} ) f \right \|_2 \leq \eps \| f \|_{X^\alpha},
\end{equation}
where we set
$$
A_{s, \lambda} := \frac{ (-\DD)^{\frac{s_\alpha}{2}}+1}{(-\DD)^s + \lambda} . 
$$
Next, we note that $A_{\tilde{s},\tilde{\lambda}}=A_{\tilde{s},\lambda}-(\lambda-\tilde{\lambda})B_{\tilde{s},\lambda, \tilde{\lambda}}$ with
$$
B_{\tilde{s},\lambda,\tilde{\lambda}} = \frac{ (-\DD)^{\frac{s_\alpha}{2}} + 1 }{ ( (-\Delta)^{\tilde{s}}+\tilde{\lambda})( (-\DD)^{\tilde{s}} + \lambda) } .
$$
Furthermore, we observe that
$$
A_{s,\lambda} |Q|^\alpha - A_{\tilde{s},\lambda} |\tilde{Q}|^\alpha
= \big ( A_{s,\lambda} - A_{\tilde{s},\lambda} \big ) |Q|^\alpha
+ A_{s,\lambda} \big( |\tilde{Q}|^\alpha-|Q|^\alpha \big ) \,.
$$
Hence the left-hand side of \eqref{ineq:cont_FQ2} can be estimated as follows
\begin{equation*} \label{ineq:cont3}
\mbox{LHS of \eqref{ineq:cont_FQ2}} \leq I + II + III ,
\end{equation*}
where
$$
I = \left \| \big (A_{s,\lambda} - A_{\tilde{s}, \lambda} \big )|Q|^{\alpha} f \right \|_{2}, \quad 
 II= \left \|A_{s,\lambda} \big ( |\tilde{Q}|^\alpha - |Q|^\alpha \big ) f \right \|_{2}, 
$$
$$
III = \left \| (\lambda-\tilde\lambda) B_{\tilde{s}, \lambda, \tilde{\lambda}} |\tilde Q|^\alpha f  \right \|_{2}.
$$
To estimate $I$, we recall the bounds for $f(\xi, s)$ derived above and we find (with $\sigma > 0$ small such that $s_0 > s_\alpha + 2 \sigma$) the following bound
\begin{align*}
I & \leq \sup_{\xi \in \R} \left | (f(\xi, s) - f(\xi, \tilde{s})) (|\xi|^{s_\alpha} +1) (|\xi|^{s_\alpha+\sigma} +1)  \right | \left \| \frac{1}{(-\DD)^{\frac{s_\alpha+\sigma}{2}} +1 } |Q|^\alpha f \right \|_2 \\
& \lesssim_{s_0, \alpha, \sigma, \lambda} |s-\tilde{s}| \sup_{\xi \in \R} \left ( \frac{ |\xi|^{2s_\alpha + \sigma} +1}{|\xi|^{2s_0 - \sigma} +1} \right ) \| Q \|_{\alpha+2}^{\alpha} \| f \|_{\alpha+2} \\
& \lesssim_{s_0, \alpha, \sigma, \lambda} |s-\tilde{s}| \| Q \|_{X^\alpha}^\alpha \| f \|_{X^\alpha} \leq \frac{\eps}{3} \| f \|_{X^\alpha},
\end{align*} 
provided that $| s-\tilde{s} | \leq \delta$ for some $\delta > 0$. Here we also used \eqref{ineq:Fwell}.

To control $II$, we choose again $\sigma > 0$ small such that $s_0 > s_\alpha + \sigma$, which yields  
\begin{align*}
II & \lesssim_{s_0, \alpha, \sigma, \lambda} \sup_{\xi \in \R} \left | \frac{(|\xi|^{s_\alpha}+1)(|\xi|^{s_\alpha + \sigma} +1)}{|\xi|^{2s_0-\sigma} + \lambda} \right | \left \| \frac{1}{(-\DD)^{\frac{s_\alpha + \sigma}{2}} +1 }  ( |\tilde{Q}|^\alpha - |Q|^\alpha) f \right \|_2 \\
& \lesssim_{s_0, \alpha, \sigma, \lambda} \| |\tilde{Q}|^\alpha - |Q|^\alpha \|_{\frac{\alpha+2}{\alpha}}  \| f \|_{\alpha+2} .
\end{align*}
Suppose now that $0 < \alpha \leq 1$. Then $| |\tilde{Q}|^\alpha - |Q|^\alpha| \leq |\tilde{Q} - Q|^\alpha$ pointwise a.\,e.~in $\R$. On the other hand, if we have $\alpha > 1$, we deduce that $| |\tilde{Q}|^\alpha - |Q|^{\alpha} | \lesssim_\alpha (|\tilde{Q}|^{\alpha-1}+ |Q|^{\alpha-1} ) |\tilde{Q} - Q|$ pointwise a.\,e.~in $\R$. Hence, in either case, we can apply H\"older's inequality to conclude that
$$
II \lesssim_{s_0, \alpha, \sigma, \lambda, \| Q \|_{X^\alpha}} \| \tilde{Q} - Q \|_{X^\alpha}^{\min \{\alpha,1 \}} \|f \|_{X^\alpha} \leq \frac{\eps}{3} \| f \|_{X^\alpha}, 
$$
provided that $\| \tilde{Q} - Q \|_{X^\alpha} \leq \delta$ for some $\delta > 0$.

Finally, we remark that we readily deduce that
$$
III \lesssim_{s, \alpha, \lambda} |\tilde\lambda-\lambda| \| Q \|_{\alpha+2}^\alpha \| f \|_{\alpha+2} \leq \frac{\eps}{3} \| f \|_{X^\alpha},
$$
provided that $|\tilde\lambda-\lambda| \leq \delta$ for some $\delta > 0$. This completes the proof that $\frac{\partial F_1}{\partial Q}$ depends continuously on $(Q, \lambda,s)$. 

The arguments that show continuity for the derivatives $\frac{\partial F_1}{\partial \lambda}$ and $\frac{\partial F_2}{\partial s}$ are very similar to the estimates given above. Therefore we omit the details, and the proof of Lemma \ref{lem:F_C1} is now complete. \end{proof}

\end{appendix}

\bibliographystyle{siam}









\end{document}